\documentclass[11pt,reqno]{amsart}
\usepackage[english]{babel}
\usepackage{amsmath,slashed}
\usepackage{amsfonts}
\usepackage{mathtools}
\usepackage{amssymb}
\usepackage{bbm}
\usepackage{graphicx}
\usepackage[left=2cm,right=2cm,top=2cm,bottom=2cm]{geometry}
\usepackage{amsmath,tikz-cd} 

\usepackage[final]{pdfpages} 
\usepackage{subcaption} 

\newlength{\intwidth}

\DeclareRobustCommand{\Bint}
   {\mathop{%
      \text{%
        \settowidth{\intwidth}{$\int$}%
        \makebox[0pt][l]{\makebox[\intwidth]{$-$}}%
        $\int$}}}

\DeclareMathOperator\Div{div}

\usepackage{amssymb}
\usepackage[centertags]{amsmath}
\usepackage{verbatim}

\DeclareMathOperator\curl{curl}

\newcommand\Lone{\xrightarrow[\mathrm{a.s.}]{L^1}}

\newcommand{\pd}{\partial}
\newcommand\om\omega
\newcommand{\al}{\alpha}
\newcommand{\be}{\beta}
\newcommand{\si}{\sigma}
\newcommand{\Si}{\Sigma}
\newcommand{\De}{\Delta}
\def\RR{\mathbb{R}}

\renewcommand\ell{{ l }}
\DeclareMathOperator\Real{Re}
\DeclareMathOperator\Imag{Im}

\def\cN{\mathcal N}
\def\cZ{\mathcal Z}

\def\hu{\widehat u}

\def\cA{\mathcal A}
\def\TT{\mathbb{T}}
\def\S{\mathbb{S}}
\newcommand\vp\varphi
\newcommand\ka\kappa
\newcommand\te\theta
\newcommand\vka\varkappa

\DeclareMathOperator\Tr{tr}

\newcommand\de\delta
\newcommand\La\Lambda
\newcommand\la\lambda
\newcommand\ga\gamma
\newcommand\CC{\mathbb C}
\newcommand\Ga\Gamma

\newcommand\cT{\mathcal T}
\newcommand\cI{\mathcal I}
\newcommand\cJ{\mathcal J}

\newcommand\cD{\mathcal{D}}
\newcommand{\triple}[1]{{\left\vert\kern-0.25ex\left\vert\kern-0.25ex\left\vert #1
        \right\vert\kern-0.25ex\right\vert\kern-0.25ex\right\vert}}

\usepackage[linkcolor=blue,colorlinks=true, citecolor=blue]{hyperref}
\let\originaleqref\eqref 
\makeatletter
\renewcommand{\eqref}[1]{%
  \begingroup%
  \let\ref\@refstar%
  \hyperref[#1]{%
    \originaleqref{#1}%
  }%
  \endgroup
}
\makeatother

\usepackage[toc, page]{appendix}

\newcommand{\vol}[1]{|#1|}

\newcommand{\bE}{\mathbb{E}}
\newcommand{\bP}{\mathbb{P}}

\newcommand{\cX}{\mathcal{X}}

\newcommand\bPu{\mu_u}

\makeatletter
\renewcommand*{\@fnsymbol}[1]{\ensuremath{\ifcase#1\or *\or \star\or ***\or
   \mathsection\or \mathparagraph\or \|\or **\or \dagger\dagger
   \or \ddagger\ddagger \else\@ctrerr\fi}}
\makeatother

\usepackage{fancyhdr}
\newcommand\ep\varepsilon

\newtheorem{theorem}{Theorem}[section]
\newtheorem{lemma}[theorem]{Lemma}

\newtheorem{proposition}[theorem]{Proposition}
\newtheorem{corollary}[theorem]{Corollary}

\theoremstyle{definition}
\newtheorem{remark}[theorem]{Remark}

\newcommand\hF{{\widehat{F}}}
\numberwithin{equation}{section}

\renewcommand\leq\leqslant
\renewcommand\geq\geqslant

\newcommand\cz{c^{\mathrm{z}}}
\newcommand\Nz{N^{\mathrm{z}}_u}
\newcommand\NzL{N^{\mathrm{z}}_{u^L}}
\newcommand\NzLz{N^{\mathrm{z}}_{u^{L,z}}}
\newcommand\No{N^{\mathrm{o}}_u}
\newcommand\NoL{N^{\mathrm{o}}_{u^L}}
\newcommand\Nt{N^{\mathrm{t}}_u}
\newcommand\NtL{N^{\mathrm{t}}_{u^L}}
\newcommand\NteL{N^{\mathrm{t,e}}_{u^L}}
\newcommand\Vt{V^{\mathrm{t}}_u}
\newcommand\VtL{V^{\mathrm{t}}_{u^L}}
\newcommand\Nh{N^{\mathrm{h}}_u}
\newcommand\NhL{N^{\mathrm{h}}_{u^L}}

\newcommand\nuz{\nu^{\mathrm{z}}}
\newcommand\nuo{\nu^{\mathrm{o}}}
\newcommand\nut{\nu^{\mathrm{t}}}
\newcommand\nuh{\nu^{\mathrm{h}}}
\newcommand\nuhs{\nu^{\mathrm{h}}_*}

\newcommand\htop{h_{\mathrm{top}}}

\newcommand\fS{\mathfrak{S}}

\title[Knots and chaos in random Beltrami fields]{Almost sure existence of
  knots and chaos\\ in random Beltrami fields}

\title[Beltrami
  fields exhibit knots and chaos]{Beltrami
  fields exhibit\\ knots and chaos almost surely}

\author{Alberto Enciso}
\address{Instituto de Ciencias Matem\'aticas, Consejo Superior de
  Investigaciones Cient\'\i ficas, 28049 Madrid, Spain}
\email{aenciso@icmat.es}

\author{Daniel Peralta-Salas}
\address{Instituto de Ciencias Matem\'aticas, Consejo Superior de
 Investigaciones Cient\'\i ficas, 28049 Madrid, Spain}
\email{dperalta@icmat.es}

\author{\'Alvaro Romaniega}
\address{Instituto de Ciencias Matem\'aticas, Consejo Superior de
 Investigaciones Cient\'\i ficas, 28049 Madrid, Spain}
\email{alvaro.romaniega@icmat.es}

\begin{document}
\maketitle

\begin{abstract}
  In this paper we show that, with probability~$1$, a random Beltrami
  field exhibits chaotic regions that coexist with invariant tori of
  complicated topologies. The motivation to consider this question,
  which arises in the study of stationary Euler flows in dimension~3,
  is V.I.~Arnold's 1965 conjecture that a typical Beltrami field
  exhibits the same complexity as the restriction to an energy
  hypersurface of a generic Hamiltonian system with two degrees of
  freedom. The proof hinges on the obtention of asymptotic bounds for
  the number of horseshoes, zeros, and knotted invariant tori and
  periodic trajectories that a Gaussian random Beltrami field
  exhibits, which we obtain through a nontrivial extension of the
  Nazarov--Sodin theory for Gaussian random monochromatic waves and the application of different tools from the theory of dynamical systems, including KAM theory, Melnikov analysis and hyperbolicity. Our
  results hold both in the case of Beltrami fields on~$\RR^3$ and of
  high-frequency Beltrami fields on the 3-torus.
\end{abstract}

\section{Introduction}

Beltrami fields, that is, eigenfunctions of the curl operator
satisfying
\begin{equation}\label{BF}
\curl u =\la u
\end{equation}
on~$\RR^3$ or on the flat torus~$\TT^3$ for some nonzero
constant~$\la$, are a classical family of stationary solutions to the
Euler equation in three dimensions. However, the significance of
Beltrami fields in the context of ideal fluids in equilibrium was only
unveiled by V.I.~Arnold in his
influential work on stationary Euler flows. Indeed, Arnold's
structure theorem~\cite{Ar65,Ar66} ensures that, under
suitable technical assumptions, a smooth stationary solution to the 3D Euler
equation is either integrable or a Beltrami field. In the language of
fluid mechanics, an integrable flow is usually called laminar, so
complex dynamics (as expected in Lagrangian turbulence) can only
appear in a fluid in equilibrium through Beltrami fields. This
connection between Lagrangian turbulence and Beltrami fields is so
direct that physicists have even coined the term ``Beltramization'' to
describe the experimentally observed phenomenon that the velocity field and its curl
(i.e., the vorticity) tend to align in turbulent regions (see
e.g.~\cite{Farge,Monchaux}).

Motivated by H\'enon's numerical studies of ABC flows~\cite{He66},
which are the easiest examples of Beltrami fields,
Arnold conjectured~\cite{Ar65,Ar66} that Beltrami fields exhibit the
same complexity as the restriction to an energy level of a typical
mechanical system with two degrees of freedom. To put it differently,
a typical Beltrami field should then exhibit chaotic regions
coexisting with a positive measure set of invariant tori of
complicated topology.

Although specific instances of chaotic ABC flows in the nearly
integrable regime have been known for a long time~\cite{ABC}, the
conjecture is wide open. A major step towards the proof of this claim
was the construction of Beltrami fields on~$\RR^3$ with periodic
orbits and invariant tori (possibly with homoclinic
intersections~\cite{Alex} inside) of arbitrary knotted
topology~\cite{Annals,Acta}. In fluid mechanics, these periodic orbits
and invariant tori are usually called {vortex lines} and vortex tubes,
respectively, and in fact the existence of vortex lines of any
topology had also been conjectured by Arnold in the same papers. These results
also hold~\cite{ENS} in the case of Beltrami fields on~$\TT^3$,
which, contrary to what happens in the case of~$\RR^3$, have finite
energy; this is important for applications because~$\RR^3$ and~$\TT^3$
are the two main settings in which mathematical fluid mechanics is studied. The main drawback of the approach we developed to prove these
results is that, while we managed to construct structurally stable
Beltrami fields exhibiting complex behavior, the method of proof
provides no information whatsoever about to what extent complex
behavior is typical for Beltrami fields.

Our objective in this paper is to establish Arnold's view of
complexity in Beltrami fields. To do so, the key new tool is
a theory of random Beltrami fields, which we develop here in order to
estimate the probability that a Beltrami field exhibits certain
complex dynamics. The blueprint for this is the Nazarov--Sodin theory
for Gaussian random monochromatic waves, which yields asymptotic laws for the
number of connected nodal components of the wave. Heuristically, the basic idea is that a Beltrami field
satisfying~\eqref{BF} can be thought of as a vector-valued
monochromatic wave; however, the vector-valued nature of the solutions
and the fact that we aim to control much more sophisticated geometric
objects introduces essential new difficulties from the very beginning.

\subsection{Overview of the Nazarov--Sodin theory for Gaussian random
  monochromatic waves}

The Nazarov--Sodin theory~\cite{NS16}, whose original motivation was
to understand the nodal set of random spherical harmonics of large
order~\cite{NS09}, provides a very efficient tool to derive asymptotic
laws for the distribution of the zero set of smooth Gaussian functions
of several variables. The primary examples are various Gaussian
ensembles of large-degree polynomials on the sphere or on the torus
and the restriction to large balls of translation-invariant Gaussian
functions on~$\RR^d$. Most useful for our purposes are their
asymptotic results for Gaussian random monochromatic
waves, which are random solutions to the Helmholtz equation
\begin{equation}\label{Helmholtz}
\De F + F=0
\end{equation}
on~$\RR^d$. We will henceforth restrict ourselves to the case $d=3$
for the sake of concreteness.

As the Fourier transform of a solution to the Helmholtz
equation~\eqref{Helmholtz} must be supported
on the sphere of radius~1, the way one constructs random monochromatic waves
is the following~\cite{CS19}. One starts with a real-valued
orthonormal basis of the space of square-integrable functions on the
unit two-dimensional sphere~$\S$. Although the choice of basis is
immaterial, for concreteness we can think of the basis of spherical
harmonics, which we denote by~$Y_{lm}$. Hence
$Y_{lm}$ is an eigenfunction of the spherical Laplacian with
eigenvalue $l(l+1)$, the index $l$ is a non-negative integer and $m$
ranges from $-l$ to~$l$. The degeneracy of the eigenvalue $l(l+1)$ is
therefore $2l+1$. To consider a Gaussian random monochromatic
wave, one now sets
\begin{subequations}\label{random}
\begin{equation}\label{random1}
\vp(\xi)\coloneqq \sum_{l=0}^\infty\sum_{m=-l}^{l} i^l \, a_{\ell m}\,Y_{lm}(\xi)
\end{equation}
on the unit sphere  $|\xi|=1$, $\xi\in\RR^3$, where $a_{l m}$ are independent
standard Gaussian random variables. One then defines~$F$ as the Fourier transform of the measure~$\vp\, d\sigma$, where $d\sigma$ is
the area measure of the unit sphere. This is tantamount to setting
\begin{equation}\label{random2}
F(x)\coloneqq(2\pi)^{\frac 32} \sum_{l=0}^\infty\sum_{m=-l}^{l} a_{lm} \,
Y_{lm}\bigg(\frac x{|x|}\bigg) \,\frac{J_{l+\frac 12}(|x|)}{|x|^{\frac 12}}\,.
\end{equation}
\end{subequations}

The central known result concerning the asymptotic distribution of the
nodal components of Gaussian random monochromatic waves is that, almost surely,
the number of connected components of the nodal set that are contained
in a large ball (and even those of any fixed compact topology) grows
asymptotically like the volume of the ball. More precisely, let us denote by $N_F(R)$ (respectively, $N_F(R;[\Si])$) the number of connected components of the nodal
set $F^{-1}(0)$ that are contained in the ball centered at the origin
of radius~$R$ (respectively, and diffeomorphic to~$\Si$). Here $\Si$ is any smooth, closed, orientable
surface $\Si\subset\RR^3$. It is obvious from the definition that $N_F(R;[\Si])$ only depends on
the diffeomorphism class of the surface,~$[\Si]$. The main result of
the theory ---which is due to Nazarov and Sodin~\cite{NS16} in the
case of nodal sets of any topology,
and to Sarnak and Wigman when the topology of the nodal
sets is controlled~\cite{SW19}--- can then be stated as follows. Here and in what
follows, the symbol $\Lone$ will be used to denote that a certain
sequence of random variables converges both almost surely and
in mean. Morally speaking, this is a law of large numbers for
the number of connected components associated with the Gaussian field~$F$.

\begin{theorem}\label{T.NS}
Let $F$ be a monochromatic random wave. Then there are positive
constants $\nu$, $\nu([\Si])$ such that, as $R\to\infty$,
  \[
    \frac{N_F(R)}{|B_R|}\Lone \nu\,,\qquad
    \frac{N_F(R;[\Si])}{|B_R|}\Lone \nu([\Si])\,.
  \]
  Here~$\Si\subset\RR^3$ is any compact surface as above.
\end{theorem}

\subsection{Gaussian random Beltrami fields on~$\RR^3$}

Our goal is then to obtain an extension of the Nazarov--Sodin theory
that applies to random Beltrami fields. As we will discuss later in
the Introduction, this is far from trivial
because there are essential new difficulties that make the
analysis of the problem rather involved.

The origin of many of these difficulties is strongly geometric. In
contrast to the case of random monochromatic waves (or any other
scalar Gaussian field), where the main geometric objects of interest
are the components of its nodal set, in the study of random vector
fields we aim to understand structures of a much subtler geometric
nature. Among these structures, and in increasing order of complexity,
one should certainly consider the following:
\begin{enumerate}
\item {\em Zeros}\/, i.e., points where the vector field
  vanishes.

  \item {\em Periodic orbits}\/, which can be knotted in complicated
    ways.

    \item {\em Invariant tori}\/, that is, surfaces diffeomorphic to
      a 2-torus that are invariant under the flow of the field. They
      can be knotted too.

      \item {\em Compact chaotic invariant sets}\/, which exhibit
        horseshoe-type dynamics and have, in particular, positive
        topological entropy.
      \end{enumerate}
      Recall that a {horseshoe} is defined as a
    compact hyperbolic invariant set on which the time-$T$ flow of~$u$
    is topologically conjugate to a Bernoulli shift~\cite{GH},
    for some~$T$.
Consequently, let us define the following quantities:
\begin{enumerate}
\item $\Nz(R)$ denotes the number of zeros of~$u$ contained
  in the ball~$B_R$.

  \item Given a (possibly knotted) closed curve~$\ga\subset\RR^3$, $\No(R;[\ga])$ denotes
    the number of periodic orbits of~$u$ contained in~$B_R$ that are
    isotopic to~$\ga$.

\item Given a (possibly knotted)  torus~$\cT\subset\RR^3$, $\Vt(R;[\cT])$ is the volume
  (understood as the inner measure) of
  the set of ergodic invariant tori of~$u$ that are contained in~$B_R$ and are isotopic
  to~$\cT$. Ergodic means that we consider invariant tori on which the orbits of $u$ are dense.

  \item $\Nh(R)$ denotes the number of horseshoes of~$u$ contained
    in the ball~$B_R$.
  \end{enumerate}
   Clearly, these quantities only depend on the isotopy
    class of~$\ga$ and~$\cT$.

    It is not hard to believe that these geometric subtleties give
    rise to a number of analytic difficulties. One should mention,
    however, that there also appear other unexpected analytic
    difficulties whose origin is less obvious. They are related to the
    fact that it is not clear how to define a random Beltrami field
    through an analog of~\eqref{random2}. This is because the
    characterization of a monochromatic wave as the Fourier transform
    of a distribution supported on a sphere is the conceptual base of
    the simple definition~\eqref{random1}, which underlies the
    equivalent but considerably more awkward
    expression~\eqref{random2}. Heuristically, analytic difficulties
    stem from the fact that there is not such a clean formula in
    Fourier space for a general Beltrami field. This is because the three components of the Beltrami field (which are
    monochromatic waves) are not independent, so the reduction to a
    Fourier formulation with independent variables is not trivial. We
    refer the reader to Section~\ref{S.random}, where we explain in
    detail how to define Gaussian random Beltrami fields in a way that
    is strongly reminiscent of~\eqref{random2}. Later in this
    Introduction we shall also informally discuss the aforementioned
    difficulties and discuss how we manage to circumvent them using a
    combination of ideas from PDE, dynamical systems and probability

    We can now state our main result for Gaussian random Beltrami
    fields on~$\RR^3$, as defined in Section~\ref{S.random}. Let us
    emphasize that the picture that emerges from this theorem is fully
    consistent with Arnold's view of complexity in Beltrami fields;
    with probability~1, we show that a random Beltrami field is
    ``partially integrable'' in that there is a large volume of
    invariant tori, and simultaneously features many compact chaotic
    invariant sets and periodic orbits of arbitrarily complex
    topologies. This coexistence of chaos and order is
    indeed the essential feature of the restriction to an energy
    hypersurface of a generic Hamiltonian system with two degrees of
    freedom, as Arnold put it. In this direction, Corollary~\ref{C.R3} below is quite illustrative.


\begin{theorem}\label{T.R3}
  Let $u$ be a Gaussian random Beltrami field. Then:
  \begin{enumerate}
\item The topological entropy of~$u$ is positive almost surely. In
  fact, with probability~$1$,
  \[
\liminf_{R\to\infty} \frac{\Nh(R)}{|B_R|}>\nuh\,.
\]

\item With probability~$1$, the volume of ergodic invariant tori of~$u$ isotopic to a given embedded
  torus~$\cT\subset\RR^3$ and the number of periodic orbits
  of~$u$ isotopic to a given closed curve~$\ga\subset\RR^3$ satisfy the volumetric growth estimate
  \[
\liminf_{R\to\infty} \frac{\Vt(R;[\cT])}{|B_R|}>\nut([\cT])\,,\qquad \liminf_{R\to\infty} \frac{\No(R;[\ga])}{|B_R|}>\nuo([\ga])\,.
  \]

\end{enumerate}
The constants $\nuh$, $\nut([\cT])$ and $\nuo([\ga])$ above
are all positive, for any choice of the curve~$\ga$ and the
torus~$\cT$.
\end{theorem}

\begin{corollary}\label{C.R3}
With probability~$1$, a Gaussian random Beltrami field on~$\RR^3$ exhibits infinitely many horseshoes coexisting with an infinite volume of ergodic invariant tori of each isotopy type. Moreover, the set of periodic orbits contains all knot types.
\end{corollary}

\begin{remark}
  The result we prove  (see Theorem~\ref{T.R3b}) is in fact considerably stronger: we do not only
  prescribe the topology of the periodic orbits and the invariant
  tori we count, but also other important dynamical
  quantities. Specifically, in the case of periodic orbits we have control over
  the periods (which we can pick in a certain interval $(T_1,T_2)$)
  and the maximal Lyapunov exponents (which we can also pick in an interval~$(\La_1,\La_2)$). In the case of
  the ergodic invariant tori, we can control the associated
  arithmetic and nondegeneracy conditions. Details are provided in Section~\ref{S.R3}.
\end{remark}

Unlike the case of nodal set components considered in the context of
the Nazarov--Sodin theory for Gaussian random monochromatic waves, we
do not prove exact asymptotics for the quantities we study, but only
nontrivial lower bounds that hold almost surely. Without getting
technicalities at this stage, let us point out that this is related to
analytic difficulties arising fron the fact that we are dealing with
quantities that are rather geometrically nontrivial. If one considers
a simpler quantity such as the number of zeros of a Gaussian random
Beltrami field, one can obtain an asymptotic distribution law similar
to that of the nodal components of a random monochromatic wave, whose
corresponding asymptotic constant can even be computed explicitly:

\begin{theorem}\label{T.R3zeros}
  With probability~$1$, the number of zeros of a Gaussian random Beltrami field satisfies
  \[
\frac{\Nz(R)}{|B_R|}\Lone \nuz
\]
as $R\to\infty$. The constant is explicitly given by
\begin{equation}\label{nuz}
  \nuz:=
    \cz\int_{\RR^5}|Q(z)|\,
    e^{-\widetilde Q(z))}\,dz=0.00872538\dots\,,
  \end{equation}
where $\cz:= {21^{5/2}}/[{143\sqrt5\,\pi^4}]$, and $Q,\widetilde Q$ are the following homogeneous polynomials
in five variables:
\begin{align}\label{Q(z)}
Q(z)&:=
  z_1z_2^2+z_2^3-z_1^2z_4-z_1z_2z_4-z_3^2z_4+2z_2z_3z_5-z_1z_5^2\,,\\
  \widetilde Q(z) &:= \frac{189}{65} z_1^2+\frac{42}{11}(z_2^2+z_3^2)+\frac{42}{13}(z_4^2+z_1z_4+z_5^2)\,.\label{tildeQ}
\end{align}
\end{theorem}

\subsection{Random Beltrami fields on the torus}\label{S.introT}

A Beltrami field on the flat 3-torus $\TT^3:=(\RR/2\pi\mathbb Z)^3$
(or, equivalently, on the cube of~$\RR^3$ of side length~$2\pi$ with
periodic boundary conditions) is a vector field on~$\TT^3$ satisfying
the eigenvalue equation
\[
\curl v= \la v
\]
for some real number~$\la\neq0$. It is well-known (see e.g.~\cite{Adv}) that the spectrum of
the curl operator on the $3$-torus consists of the numbers of the form
$\la=\pm|k|$ for some vector with integer coefficients $k\in\mathbb
Z^3$. Restricting our attention to the case of positive eigenvalues for the sake
of concreteness, one can therefore label the eigenvalue by a
positive integer~$L$ such that $\la_L=L^{1/2}$. The multiplicity of
the eigenvalue is given by the cardinality of the corresponding set of
spatial frequencies,
\[
\cZ_L:= \{k\in\mathbb Z^3: |k|^2=L\}\,.
\]
By Legendre's three-square theorem, $\cZ_L$ is
nonempty (and therefore~$\la_L$ is an eigenvalue of the curl operator)
if and only if~$L$ is not of the form $4^a(8b+7)$ for nonnegative
integers~$a$ and~$b$.

The Beltrami fields corresponding to the eigenvalue~$\la_L$ must
obviously be of
the form
\begin{equation*}
u^L=\sum_{k\in\cZ_L} V_k^L\, e^{ik\cdot x}\,,
\end{equation*}
for some vectors $V_k^L\in\CC^3$, where $V_k^L=
\overline{V_{-k}^L}$ to ensure that the Beltrami field is
real-valued. Starting from this formula, in Section~\ref{S.BFtorus} we
define the Gaussian ensemble of random Beltrami
fields~$u^L$ of frequency~$\la_L$, which we parametrize by~$L$. The
natural length scale of the problem is~$L^{1/2}$.

Our objective is to study to what extent the appearance of the various
dynamical objects described above (i.e., horseshoes, zeros, and
periodic orbits and ergodic invariant tori of prescribed topology) is
typical in high-frequency Beltrami fields, which corresponds to the
limit $L\to\infty$. When taking this limit, we shall always assume
that the integer~$L$ is {\em admissible}\/, by which we mean that it
is congruent with 1, 2, 3, 5 or 6 modulo~8. We will see in
Section~\ref{S.BFtorus} (see also~\cite{Roz17}) that this number-theoretic condition ensures
that the dimension of the space of Beltrami fields with
eigenvalue~$\la_L$ tends to infinity as $L\to\infty$.

To state our main result about high-frequency random Beltrami fields
in the torus we need to introduce some notation. In parallel with the
previous subsection, for any closed curve~$\ga$ and any embedded
torus~$\cT$, let us respectively denote by $\NzL$, $\NhL$,
$\NoL([\ga])$ and $\NtL([\cT])$ the number of zeros, horseshoes,
periodic orbits isotopic to~$\ga$ and ergodic invariant tori isotopic
to~$\cT$ of the field~$u^L$, as well as the volume (i.e., inner
measure) of these tori, which we denote by $\VtL([\cT])$. To further
control the distribution of these objects, let us define the
number of approximately equidistributed ergodic invariant tori,
$\NteL([\cT])$, as the largest
integer~$m$ for which $u^L$ has~$m$ ergodic invariant tori isotopic
to~$\cT$ that are at a distance greater than~$m^{-1/3}$ apart from one
another. The
number of approximately equidistributed horseshoes
$N^{\mathrm{h},\mathrm{e}}_{u^L}$, periodic orbits isotopic to a curve
$N^{\mathrm{o},\mathrm{e}}_{u^L}([\ga])$ and zeros
$N^{\mathrm{z},\mathrm{e}}_{u^L}$ are defined analogously. Note that, again, the asymptotic information that we obtain is perfectly aligned
with Arnold's view of complex behavior in typical Beltrami fields.

\begin{theorem}\label{T.torus}
  Let us denote by $(u^L)$ the parametric Gaussian ensemble of random
  Beltrami fields on~$\TT^3$, where~$L$ ranges over the set of admissible
  integers. Consider any contractible closed curve~$\ga$ and any
  contractible embedded torus~$\cT$ in~$\TT^3$. Then:
  \begin{enumerate}
  \item With a probability tending to~$1$ as $L\to\infty$, the
    field~$u^L$ exhibits an arbitrarily large number of approximately distributed horseshoes,
    zeros, periodic orbits isotopic to~$\ga$ and ergodic invariant
    tori isotopic to~$\cT$. More precisely, for any integer~$m$,
\[
  \lim_{L\to \infty} \bP\Big\{\min\big\{N^{\mathrm{h},\mathrm{e}}_{u^L},\NteL([\cT]),N^{\mathrm{o},\mathrm{e}}_{u^L}([\ga]),N^{\mathrm{z},\mathrm{e}}_{u^L}\big\} >m\Big\}=1\,.
\]
Furthermore, the probability that the topological entropy
      of the field grows at least as~$L^{1/2}$ and that there are infinitely many ergodic invariant
      tori of~$u^L$ isotopic to~$\cT$ also tends to~$1$:
    \[
\lim_{L\to \infty} \bP\big\{\NtL([\cT])=\infty\; \text{ and } \; \htop(u^L)> \nuhs L^{1/2}\big\}=1\,.
      \]

    \item The  expected
      volume of the ergodic invariant tori of~$u^L$ isotopic to~$\cT$ is
      uniformly bounded from below, and the expected number of
      horseshoes and periodic orbits isotopic to~$\ga$ is at
      least of order~$L^{3/2}$:
      \begin{align*}
        \liminf_{L\to\infty} \min\Bigg\{\frac{\bE\NhL}{L^{3/2}}\,,
        \frac{\bE\NoL([\ga])}{L^{3/2}}\,,  \bE \VtL([\cT])\Bigg\}>\nu_*([\ga],[\cT])\,.
      \end{align*}
       In the case of zeros, the asymptotic expectation is explicit,
       with $\nuz$ given by~\eqref{nuz}:
      \[
\lim_{L\to\infty}\frac{\bE\NzL}{L^{3/2}}=  (2\pi)^3\nuz\,.
\]
  \end{enumerate}
  Here $\nuhs$ and $\nu_*([\ga],[\cT])$ are positive constants.
\end{theorem}

\begin{remark}
  As in the case of~$\RR^3$, the result we prove in
  Section~\ref{S.BFtorus} is actually stronger in the
  sense that we have control over important dynamical quantities
  (which now depend strongly on~$L$)
  describing the flow near the above invariant tori and periodic orbits.
\end{remark}

\subsection{Some technical remarks}

In a way, the cornerstone of the Nazarov--Sodin theory is their very
clever (and non-probabilistic) ``sandwich estimate'', which relates the number $N_F(R)$ of
connected components of the nodal set of the Gaussian random field~$F$
that are contained in an
arbitrarily large ball~$B_R$ with ergodic averages of the same quantity
involving the number of components contained in balls of fixed
radius. Two ingredients are key to effectively apply this sandwich estimate. On the one hand, each nodal
component cannot be too small by the Faber--Krahn inequality, which
ensures, in dimension~3, that its volume is at least $c\la^{-3}$ if
$\De F+\la^2F=0$. On the other hand, to control the
connected components that intersect a large ball but are not contained
in it, it suffices to employ the Kac--Rice formula to derive bounds for the number of critical points
of a certain family of Gaussian random functions.

In the setting of random Beltrami fields, the need for new ideas
becomes apparent the moment one realizes that there are no reasonable
substitutes for these two key ingredients. That is,
the frequency~$\la$ does not provide bounds for the size of the more
sophisticated geometric objects considered in this context (i.e., periodic
orbits, invariant tori or horseshoes), and one cannot estimate the
objects that intersect a ball but are not contained in it using a
Kac--Rice formula. As a matter of fact, we have not managed to obtain
any useful bounds for these quantities and, while we do use a sandwich
inequality of sorts (or at least lower bounds that can be regarded as a
weaker substitute thereof), even the measurability of the various
objects of interest becomes a nontrivial issue due to their
complicated geometric properties.

To circumvent these problems, we employ different kinds of
techniques. Firstly, ideas from the theory of dynamical systems play a
substantial role in our proofs. On the one hand, KAM theory
and hyperbolic dynamics are important to prove that certain
carefully chosen functionals are lower semicontinuous, which is key to
solve measurability issues that would be very hard to deal with
otherwise. Furthermore, to prove that Beltrami fields exhibit chaotic
behavior almost surely, it is essential to have at least one example
of a Beltrami field that features a horseshoe, and even that was not
known. Indeed, the available examples of non-integrable ABC flows are
known to be chaotic on~$\TT^3$ due to the non-contractibility of the
domain, but not on~$\RR^3$. This technical point is
fundamental, and makes them unsuitable for the study of
random Beltrami fields. Therefore, an important step in our proof is
to construct, using Melnikov theory, a Beltrami field on~$\RR^3$ that
has a horseshoe. Techniques from Fourier analysis and from the global
approximation theory for Beltrami fields are also necessary to handle the
inherent difficulties that stem from the fact that the equation under
consideration is more complicated than that of a monochromatic
wave. As an aside, the only point of the paper where we use the
Kac--Rice formula is to compute the constant~$\nuz$ in closed form.

In the case of Beltrami fields on the torus, the results we prove
concern not only the expected values of the quantities of interest,
but also the probability of events. In the case of random
monochromatic waves on the torus, Nazarov and Sodin~\cite{NS16} had
proved results for the expectation (which apply to very general
parametric scalar Gaussian ensembles), and Rozenshein~\cite{Roz17} had
derived very precise exponential bounds for the probability akin to
those established by Nazarov and Sodin~\cite {NS09} for random
spherical harmonics. However, both results use in a crucial way that
the size of nodal components can be effectively estimated in terms of
the frequency: the Faber--Krahn inequality provides a lower bound for
the volume and large diameter components can be ruled out using a
Crofton-type formula and B\'ezout's theorem. No such bounds hold in
the case of Beltrami fields, so the way we pass
from the information that the rescaled covariant kernel of~$u^L$ tends
to that of~$u$ to asymptotics for the distribution of invariant tori,
horseshoes or periodic orbits is completely different. Specifically,
we rely on a direct argument ensuring the
weak convergence of sequences of probability measures, on spaces of
smooth functions, provided that suitable tightness conditions are satisfied.

\subsection{Outline of the paper} In Section~\ref{S.Fourier}, we start
by describing Beltrami fields in~$\RR^3$ from the point of view of
Fourier analysis and provide some results about global
approximation. Gaussian random Beltrami fields on~$\RR^3$ are
introduced in Section~\ref{S.random}, where we also establish several
results about the structure of the corresponding covariance matrix and
about the induced probability measure on the space of smooth vector
fields. In Section~\ref{S.preliminaries} we recall, in a form that
will be useful in later sections, several previous results about
ergodic invariant tori and periodic orbits arising in Beltrami
fields. Section~\ref{S.chaos} is devoted to constructing a Beltrami
field on~$\RR^3$ that is stably chaotic. Finally, in Sections~\ref{S.R3}
and~\ref{S.BFtorus} we complete the proofs
of our main results in the case of~$\RR^3$ and~$\TT^3$,
respectively. The paper concludes with an Appendix where we provide a
fairly complete Fourier-theoretic characterization of Beltrami fields.

\section{Fourier analysis and approximation of Beltrami fields}
\label{S.Fourier}

In what follows, we will say that a vector field $u$ on~$\RR^3$ is a
Beltrami field if
\[
\curl u = u\,.
\]
Taking the curl of this equation and using that necessarily $\Div
u=0$, it is easy to see that~$u$ must also satisfy the Helmholtz equation:
\[
\De u +u=0\,.
\]
To put it differently, the components of this vector field are
monochromatic waves. An immediate consequence of this is that the
Fourier transform~$\hu$ of a polynomially bounded Beltrami field is a
(vector-valued) distribution supported on the unit sphere
\[
\S:=\{\xi\in\RR^3:|\xi|=1\}\,.
\]
Since~$u$ is real-valued, $\hu$ must be Hermitian, i.e., $\hu(\xi)=
\overline{\hu(-\xi)}$. Furthermore, a classical result due to
Herglotz~\cite[Theorem 7.1.28]{Hor15} ensures that if~$u$ is a Beltrami field
with the sharp fall off at infinity, then there is a Hermitian vector-valued function $f\in L^2(\S,\CC^3)$ such that
$\hu = f\, d\si$; for the benefit of the reader, details on this and other related matters are
summarized in Appendix~\ref{A.Fourier}. For short, we shall simply write this relation as
$u=U_f$, with
\begin{equation}\label{uf}
U_f(x) :=\int_\S f(\xi)\, e^{i\xi\cdot x}\, d\si(\xi)\,.
\end{equation}
Obviously $U_f$ is a Beltrami field if and only if~$f$ is
Hermitian (which makes~$U_f$ real valued) and if it satisfies the
distributional equation on the sphere
\begin{equation}\label{eqf}
i\xi\times f(\xi)=f(\xi)\,.
\end{equation}

In this paper, we are particularly interested in Beltrami fields of
the form $u=U_f$, where now $f$ is a general Hermitian vector-valued
distribution on the sphere.
The corresponding integral, which is convergent if~$f$ is integrable,
must be understood in the sense of distributions for less regular~$f$
(that is to say, for $f$ in the scale of Sobolev spaces $H^s(\S,\CC^3)$ with
$s<0$). We recall, in particular, that for any integer $k\geq0$ the
field~$U_f$ is bounded as~\cite[Appendix~A]{random}
\begin{equation}\label{bounduu}
\sup_{R>0} \frac1R\int_{B_R}\frac{|U_f(x)|^2}{1+|x|^{2k}}\, dx\leq C\|f\|_{H^{-k}(\S,\CC^3)}\,.
\end{equation}
We recall that, for any real~$s$, the $H^s(\S)$~norm
of a function~$f$ can be computed as
\[
\|f\|_{H^s(\S)}^2=\sum_{l=0}^\infty\sum_{m=-l}^{l} (l+1)^{2s}|f_{l m}|^2\,,
\]
where $f_{lm}$ are the coefficients of the spherical harmonics
expansion of $f$.

With $q(t):= \frac18(\frac{15}{\pi})^{1/2}(1+\sqrt 7 i\, t)$, let us consider the vector-valued polynomial
\begin{equation}\label{p(xi)}
p(\xi):=q(\xi_1)\, (\xi_1^2-1,\xi_1\xi_2-i\xi_3,\xi_1\xi_3+i\xi_2)\,,
\end{equation}
which we will regard as a Hermitian function $p:\RR^3\to\CC^3$. Note
that the restriction of $p$ to the sphere vanishes exactly at the
poles $\xi_\pm:=(\pm1,0,0)$. The inessential nonvanishing
normalization factor $q(\xi_1)$ has been introduced for later
convenience: when we define random Beltrami fields via the
function~$p$ in Section~\ref{S.random}, this choice of~$p$ will ensure
that the associated covariance matrix is the identity on the diagonal
(see Corollary~\ref{C.diag}).

We next show
that, away from the poles, the density~$f$ of a Beltrami field~$U_f$
must point in the same direction as~$p$:

\begin{proposition}\label{P.p(xi)}
  The following statements hold:
  \begin{enumerate}
  \item If the vector field~$U_f$ is a Beltrami field, then
    $p\times f=0$ as a distribution on~$\S$. Furthermore, if $\chi$ is a
    smooth real-valued function on the sphere supported in
    $\S\backslash\{\xi_+,\xi_-\}$ and $f\in H^s(\S,\CC^3)$ for some real~$s$, then there is a Hermitian scalar function
    $\vp\in H^s(\S)$ such that $\chi\, f= \vp\, p$.
\item Conversely, for any Hermitian $\vp\in
H^s(\S)$, the associated field~$U_{\vp p}$
is a Beltrami field.
\end{enumerate}
\end{proposition}

\begin{proof}
In view of Equation~\eqref{eqf}, for each vector~$\xi\in\S$, consider the linear map $M_\xi$ on $\CC^3$
defined as
\[
M_\xi V:=V-i\xi\times V\,.
\]
More explicitly, $M_\xi$ is the matrix
$$
M_\xi=\left(
\begin{array}{ccc}
-1& -i \xi_3 & i \xi_2 \\
i \xi_3 & -1 & -i \xi_1 \\
-i \xi_2& i \xi_1 &-1 \\
\end{array}
\right)\,.
$$
The determinant of this matrix is $\det
M_\xi=\xi_1^2+\xi_2^2+\xi_3^2-1$, and in fact it is easy to see that
$M_\xi$ has rank~2 for any unit vector~$\xi$. Since $M_\xi
p(\xi)=0$ for all $\xi\in\S$ and $p(\xi)$ only vanishes if $\xi=\xi_\pm$, we
then obtain that the kernel of~$M_\xi$ is spanned by the
vector~$p(\xi)$ whenever~$\xi$ is not one of the poles~$\xi_\pm$. In a neighborhood of the
poles, the kernel of~$M_\xi$ can be described as the linear span of $\widetilde
p(\xi):= q(\xi_2)\, (\xi_1\xi_2+i\xi_3,\xi_2^2-1, \xi_2\xi_3-i \xi_1)$.

Since $M_\xi f(\xi)=0$ in the sense of distributions by~\eqref{eqf}, it stems from
the above analysis that one can write
\[
f(\xi)= \al(\xi)\, p(\xi)
\]
for $\xi$ away from the poles, and
\[
f(\xi)= \beta(\xi)\, \widetilde p(\xi)
\]
in a neighborhood of the poles; here $\al$ and~$\beta$ are complex-valued scalars. As $p(\xi)\times \widetilde
p(\xi)=0$ for all~$\xi\in\S$, we immediately infer that
\[
p\times f=0\,.
\]
Also, as the support of a function is a closed set, $p$ is bounded away from zero on the support
of~$\chi$, so we have that
\[
\vp:= \chi \frac{f\cdot p}{|p|^2}\in H^s(\S)\,.
\]
As $f$ is Hermitian, this proves the first part of the proposition. The second statement
follows immediately from the fact that
\[
M_\xi[\vp(\xi)p(\xi)]= \vp(\xi)\,M_\xi p(\xi)=0\,.
\]
\end{proof}

\begin{remark}\label{R.cucu}
A Beltrami field of the form $U_{\vp p}$ can be written in terms of
the scalar function $\psi(x):= -\int_\S e^{i\xi\cdot x} q(\xi_1)\, \vp(\xi)\,
d\si(\xi)$ (which satisfies the equation $\De\psi+\psi=0$) as
\[
  U_{\vp p}=(\curl \curl  +\curl) (\psi,0,0)\,.
\]
Also, it
has the sharp decay bound $|U_{\vp p}(x)|\leq C\|\vp\|_{L^2(\S)}/(1+|x|)$.
\end{remark}

\begin{remark}
Not any Beltrami field of the form~$U_f$ can be written
as $U_{\vp p}$ for some scalar function~$\vp$: an obvious
counterexample is given by
\begin{equation}\label{counterex}
  f(\xi):= (0,1,i)\, \de_{\xi_+ }(\xi) + (0,1,-i) \, \de_{\xi_- }(\xi) \,,
\end{equation}
where $\de_{\xi_\pm}$ is the Dirac measure supported on the
pole~$\xi_\pm=(\pm1,0,0)$. The reason for which we cannot hope to
describe all Beltrami fields using just scalar multiples of a fixed
complex-valued continuous vector field~$p'$ is topological. Indeed,
as~$u$ is divergence-free, we have that $\xi\cdot p'(\xi)=0$, so $p'$
must be a tangent complex-valued vector field on~$\S$. By the hairy ball theorem, the
real part of~$p'$ must then have at least one zero~$\xi^*$.  The
equation $i\xi\times p'(\xi)= p'(\xi)$ implies that the imaginary part
of~$p'$ also vanishes at~$\xi^*$, so in fact $p'(\xi^*)=0$. This means
that densities~$f$ such
as~\eqref{counterex}, where we can take $\xi^*:= \xi_+$ without any
loss of generality, cannot be written in the form~$\vp p'$.
\end{remark}

Intuitively speaking, Proposition~\ref{P.p(xi)} means that any
Beltrami field~$U_f$ whose density~$f$ is not too concentrated
on~$\xi_\pm$ can be approximated globally by a field of the
form~$U_{\vp p}$. More precisely, one can prove the following:

\begin{proposition}\label{P.approx1}
Consider a Hermitian vector-valued distribution $f$ on~$\S$ that satisfies the distributional equation~\eqref{eqf}, and define
\[
\ep_{f,k}:= \inf \big\{\|\Theta f\|_{H^{-k}(\S)} : \Theta\in
C^\infty(\S),\; \Theta(\xi_+)=\Theta(\xi_-)=1\big\}\,.
\]
If $\ep_{f,k}$ is finite and $\ep>\ep_{f,k}$, one can then take a
Hermitian scalar distribution on the sphere $\vp$, which is in fact
a finite linear combination of spherical harmonics if $f\in H^{-k}(\S,\CC^3)$, such
that
\[
\sup_{R>0}\frac1R\int_{B_R} \frac{|U_f(x)-U_{\vp p}(x)|^2}{1+|x|^{2k}}\, dx< C\ep\,.
\]
Furthermore, $\ep_{f,0}=0$ if $f\in L^2(\S,\CC^3)$.
\end{proposition}

\begin{proof}
The first assertion is a straightforward consequence of the first part of Proposition~\ref{P.p(xi)}
and of the estimate~\eqref{bounduu}. Indeed, since $f$ is a compactly supported distribution, then $f\in H^s(\S,\CC^3)$ for some~$s$. Take
any $\ep'\in(\ep_{f,k},\ep)$ and let us consider a
function~$\Theta$ as above such that $\|\Theta f\|_{H^{-k}(\S)} <\ep'$. Since $\ep'>\ep_{f,k}$, it is obvious that we can assume that $\Theta=1$ in a small neighborhood of the poles $\xi_\pm$. Applying Proposition~\ref{P.p(xi)} we infer that $\chi f=\vp p$ with
$\chi:=1-\Theta$ and some Hermitian scalar function~$\vp\in H^s(\S)$. In view of
the fact that the map $f\mapsto U_f$ is linear and of
the bound~\eqref{bounduu}, we then have
\[
\sup_{R>0}\frac1R\int_{B_R} \frac{|U_f(x)-U_{\vp
    p}(x)|^2}{1+|x|^{2k}}\, dx=\sup_{R>0}\frac1R\int_{B_R}
\frac{|U_{\Theta f}(x)|^2}{1+|x|^{2k}}\, dx \leq C\|\Theta f\|_{H^{-k}(\S,\CC^3)}<C\ep'\,.
\]

As finite linear combinations of spherical harmonics are dense in $H^s(\S)$, if $s=-k$ we
can approximate $\vp$ in the~$H^{-k}( \S)$~norm by a Hermitian function~$\vp'$
of this form; then
\begin{multline*}
\sup_{R>0}\frac1R\int_{B_R} \frac{|U_f(x)-U_{\vp'
    p}(x)|^2}{1+|x|^{2k}}\, dx\\
\leq \sup_{R>0}\frac1R\int_{B_R}
\frac{|U_{f}(x)-U_{\vp p}(x)|^2}{1+|x|^{2k}}\, dx+ \sup_{R>0}\frac1R\int_{B_R}
\frac{|U_{(\vp'-\vp)p}(x)|^2}{1+|x|^{2k}}\, dx <C\ep
\end{multline*}
provided that $\|\vp-\vp'\|_{H^{-k}(\S)}<\ep-\ep'$.

Finally, to see that $\ep_{f,0}=0$ if $f\in
L^2(\S,\CC^3)$, let us take a smooth function $\Theta:\RR^3\to[0,1]$ supported
in the unit ball and such that $\Theta(0)=1$. Setting
\[
  \Theta_n(\xi):= \Theta(n \xi- n\xi_+) + \Theta(n \xi- n\xi_-)\,,
\]
we trivially get that  $\|\Theta_n f\|_{L^2(\S)}\leq
\|f\|_{L^2(\S)}$ for all $n\geq2$ and that $\Theta_n f$ tends to zero
almost everywhere in~$\S$ as $n\to\infty$. The dominated convergence theorem then
shows that $\|\Theta_n f\|_{L^2(\S)}\to0$ as $n\to\infty$, thus proving
the claim.
\end{proof}

Another, rather different in spirit,
formulation of the principle that densities of the form~$\vp p$ can
approximate general Beltrami fields is presented in the following
theorem. Unlike the previous corollary, the approximation is
considered only locally in space, and in this direction one shows that even
considering smooth functions~$\vp$ is enough to obtain a subset of
Beltrami fields that is dense in the $C^k$ compact-open topology:

\begin{proposition}\label{P.approx}
Fix any positive reals~$\ep$ and~$k$ and a compact set
$K\subset\RR^3$ such that $\RR^3\backslash K$ is connected. Then, given any vector field~$v$ satisfying the
equation~$\curl v=v$ in an open neighborhood of~$K$, there exists a
Hermitian finite linear
combination of spherical harmonics~$\vp$ such that the Beltrami field~$U_{\vp p}$ approximates~$v$ in the set~$K$ as
\[
\|U_{\vp p}-v\|_{C^k(K)}<\ep\,.
\]
\end{proposition}

\begin{proof}
Let us fix an open set $V\supset K$ whose closure is contained in the open neighborhood where $v$ is defined, and a large ball $B_R\supset \overline V$. Since $\mathbb R^3\backslash K$ is connected, it is obvious that we can take $V$ so that $\mathbb R^3\backslash \overline V$ is connected as well. By the approximation theorem
with decay for Beltrami fields~\cite[Theorem 8.3]{Acta}, there is a Beltrami field~$w$ that
approximates~$v$ as
\[
\|w-v\|_{C^k(V)}<\ep
\]
and is bounded as $|w(x)|<C/|x|$. As the Fourier transform of~$w$ is
supported on~$\S$, Herglotz's theorem~\cite[Theorem 7.1.28]{Hor15}
shows that one can write $w=U_f$ for some vector-valued Hermitian field $f\in
L^2(\S,\CC^3)$ that satisfies the distributional equation~\eqref{eqf}. Proposition~\ref{P.approx1} then shows that there exists
some Hermitian scalar
function $\vp\in C^\infty(\S)$ such that
\[
\|U_f-U_{\vp p}\|_{L^2(B_R)}<C\ep\,,
\]
so that $\|v-U_{\vp p}\|_{L^2(V)}<C\ep$.
As the difference $v-U_{\vp p}$ satisfies the Helmholtz equation
\[
\De (v-U_{\vp p}) + v-U_{\vp p}=0
\]
in $V$, and $K\subset\!\subset V$, standard elliptic
estimates then allow us to promote this bound to
\[
\|v-U_{\vp p}\|_{C^k(K)}<C\ep\,,
\]
as we wished to prove.
\end{proof}

\section{Gaussian random Beltrami fields}
\label{S.random}

The Fourier-theoretical characterization of Beltrami fields presented
in the previous section paves the way to the definition of random
Beltrami fields.

In parallel with~\eqref{random1} (see Appendix~\ref{A.Fourier} for
further heuristics), let us start by setting
\[
\vp(\xi) := \sum_{l=0}^\infty\sum_{m=-l}^{l} i^l \, a_{\ell m}\,Y_{lm}(\xi)\,,
\]
where $a_{l m}$ are normally distributed independent
standard Gaussian random variables and $Y_{lm}$ is an orthonormal
basis of (real-valued) spherical harmonics on~$\S$. Note that $\vp$ is Hermitian
because of the identity $Y_{lm}(-\xi)=(-1)^ l Y_{lm}(\xi)$. We now
define a Gaussian random Beltrami field as
\[
u:= U_{\vp p}\,,
\]
where we recall that $U_f$ and~$p$ were respectively defined
in~\eqref{uf} and~\eqref{p(xi)}.

\begin{remark}
As discussed in Proposition~\ref{P.p(xi)}, the role of the vector field~$p$ is to ensure that the density
$f:=\vp p$ satisfies the Beltrami equation in Fourier space,
$i\xi\times f(\xi)=f(\xi)$. Hence one could replace $p(\xi)$ by any nonvanishing
multiple of it, that is, by $\widetilde p(\xi):= \La(\xi)\, p(\xi)$ where
$\La:\RR^3\to\CC$ is a smooth scalar Hermitian function that does
not vanish on~$\S$. All the results of the paper about random Beltrami
fields remain valid if one
defines a Gaussian random Beltrami field as $u:= U_{\vp \widetilde p}$
with~$\vp$ as above, provided that one replaces~$p$ by~$\widetilde p$
in the formulas. Also, the results do not change if one replaces the
basis of spherical harmonics by any other orthonormal basis
of~$L^2(\S)$, but this choice leads to slightly more explicit formulas
for certain intermediate objects that appear in the proofs.
\end{remark}

In what follows, we will use the notation $D:= -i\nabla$. An important
role will be played by the
vector-valued differential
operator with real coefficients $p(D)$, whose expression in Fourier space is
\[
\widehat{p(D)\psi}(\xi)=p(\xi)\, \widehat\psi(\xi)\,,
\]
for any scalar function $\psi$ in $\mathbb R^3$. Equivalently, by Remark~\ref{R.cucu}, the operator $p(D)$ reads, in physical space, as
\[
p(D)\psi=-(\curl \curl+\curl)(q(D_1)\psi,0,0)\,,
\]
where $D_1:= -i\pd_{x_1}$.

The first result of this section shows that a Gaussian random Beltrami field is a well defined object both in Fourier and physical spaces:
\begin{proposition}\label{P.conv}
  With probability~$1$, the function $\vp$ is in
  $H^{-1-\de}(\S)\backslash L^2(\S)$ for any $\de>0$. In particular,
  almost surely, $u$ is a $C^\infty$~vector field and can be written as
  \begin{equation}\label{seriesu}
u(x)=(2\pi)^{\frac{3}{2}}\sum_{l=0}^\infty\sum_{m=-l}^{l}  \, a_{\ell
  m}\, p(D)\left[Y_{lm} \left(\frac{x}{|x|}\right) \frac{J_{\ell + \frac12}(|x|)}{|x|^{1/2}}\right]\,.
\end{equation}
The series converges in~$C^k$ uniformly on compact sets almost surely,
for any~$k$.
\end{proposition}

\begin{proof}
  For $l\geq0$ and $-l\leq m\leq l$, $a_{lm}^2$ are independent,
  identically distributed random variables with expected value~1. As the number
  of these variables with $l\leq n$ is
  \[
\sum_{l=0}^n\sum_{m=-l}^{l}1=(n+1)^2\,,
\]
the strong law of large numbers ensures that the sample average, i.e.,
the random variable
\[
X_n:=\frac1{(n+1)^2}\sum_{l=0}^n\sum_{m=-l}^{l}a_{\ell m}^2\,,
\]
converges to~1 almost surely as $n\to\infty$. Now consider the truncation
  \[
\vp_n(\xi) := \sum_{l=0}^n\sum_{m=-l}^{l} i^l \, a_{\ell m}\,Y_{lm}(\xi)\,.
  \]
  As the spherical harmonics $Y_{lm}$ are orthonormal, the $L^2$~norm of $\vp_n$ is
  \[
\|\vp_n\|_{L^2(\S)}^2=\sum_{l=0}^n\sum_{m=-l}^{l}a_{\ell m}^2= (n+1)^2 X_n\,,
  \]
and $\|\vp_n\|_{L^2(\S)}^2$ tends to~$\|\vp\|_{L^2(\S)}^2$ (which may
be infinite) as $n\to\infty$. Since
  $X_n\to 1$ almost surely, we obtain from the above formula that $(n+1)^{-2}\|\vp_n\|_{L^2(\S)}^2$
tends to~1 almost surely. Therefore, $\vp$ is not in $L^2(\S)$ with
probability~1.

On the other hand, since
\[
\|\vp\|_{H^{-s}(\S)}^2=\sum_{l=0}^\infty\sum_{m=-l}^{l}\frac{a_{\ell m}^2}{(l+1)^{2s}}\,,
\]
it is straightforward to see that the expected value
\[
\bE\|\vp\|_{H^{-1-\de}(\S)}^2=\sum_{l=0}^\infty\sum_{m=-l}^{l}\frac{\bE a_{\ell
    m}^2}{(l+1)^{2+2\de}}= \sum_{l=0}^\infty\frac{2l+1}{(l+1)^{2+2\de}}
\]
is finite for all~$\de>0$. Hence $\vp\in H^{-1-\de}(\S)$ almost
surely, so $u:= U_{\vp p}$ is well defined with probability~1.

To prove the representation formula for~$u$ and its convergence, let
us begin by noting that
\begin{align*}
U_{i^l Y_{lm} p}(x)&= \int_\S i^l p(\xi) \,Y_{lm}(\xi)\, e^{i\xi\cdot x}\,
                   d\si(\xi)\\
  &=  p(D)\int_\S i^l  Y_{lm}(\xi)\, e^{i\xi\cdot x}\,
                   d\si(\xi)\,.
\end{align*}
Using either the theory of point pair invariants and zonal spherical
functions~\cite[Proposition~4]{CS19} or special function identities~\cite[Proposition~2.1]{random}, the Fourier transform of $Y_{lm}\, d\si$ has been shown to be
\[
\int_\S i^l  Y_{lm}(\xi)\, e^{i\xi\cdot x}\,
                   d\si(\xi)= (2\pi)^{\frac32} Y_{lm} \left(\frac{x}{|x|}\right) \frac{J_{\ell + \frac12}(|x|)}{|x|^{1/2}}\,.
\]
This permits to formally write~$u$ as~\eqref{seriesu}. To show that
this series converges in~$C^k$ on compact sets, for any large~$n$, any
$N>n$ and any fixed positive integer~$k$ consider the quantity
\[
q_{n,N}(x):= \sum_{|\al|\leq k} \left|\sum_{l=n}^N \sum_{m=-l}^l a_{lm} D^\al p(D)\left[Y_{lm} \left(\frac{x}{|x|}\right) \frac{J_{\ell + \frac12}(|x|)}{|x|^{1/2}}\right]\right|\,,
\]
where we are using the standard multiindex notation.
Since $p(D)$ is a third-order operator, for all $|x|<R$ we obviously have
\begin{align*}
q_{n,N}&(x)\leq C_k \sum_{l=n}^N \sum_{m=-l}^l
|a_{lm}|\|Y_{lm}\|_{C^{k+3}(\S)}\left\|
  \frac{J_{\ell + \frac12}(r)}{r^{1/2}}\right\|_{C^{k+3}((0,R))}\\
  &\leq C_k \left(\sum_{l=n}^N \sum_{m=-l}^l
\frac{a_{lm}^2}{(l+1)^{2+2\de}}\right)^{\frac12}\left(\sum_{l=n}^N
    \sum_{m=-l}^l(l+1)^{2+2\de} \|Y_{lm}\|_{C^{k+3}(\S)}^2 \left\|
  \frac{J_{\ell + \frac12}(r)}{r^{1/2}}\right\|^2_{C^{k+3}((0,R))}\right)^{\frac12}
\end{align*}
where here $r:=|x|$ and we have used the Cauchy--Schwartz inequality
to pass to the second line. The Sobolev inequality immediately gives
\[
\|Y_{lm}\|_{C^{k+3}(\S)}\leq C\|Y_{lm}\|_{H^{k+5}(\S)}\leq C(l+1)^{k+5}\,.
\]
To estimate the Bessel function, recall the large-degree asymptotics
\[
J_\nu(r)\sim (2\pi\nu)^{-\frac12} \left(\frac{er}{2\nu}\right)^\nu\,,
\]
which holds as $\nu\to\infty$ for uniformly bounded~$r$. As the derivative of a Bessel function can be
written in terms of Bessel functions via the recurrence relation
\[
\frac d{dr}J_\nu(r)= -J_{\nu+1}(r) + \frac\nu r J_\nu(r)\,,
\]
it follows that the $C^{k+3}$~norm of $J_{l+\frac12}(r)/r^{1/2}$ tends
to~0 exponentially as~$l\to\infty$ on compact sets:
\[
\left\|
  \frac{J_{\ell + \frac12}(r)}{r^{1/2}}\right\|_{C^{k+3}((0,R))}\leq \left(\frac{CR}{l}\right)^{l-k-3}\,.
\]
Since we have proven that
\[
\sum_{l=0}^\infty\sum_{m=-l}^l
\frac{a_{lm}^2}{(l+1)^{2+2\de}}<\infty
\]
almost surely, now one only has to put together the estimates above to
see that, almost surely, $q_{n,N}(x)$ tends to~0 as
$n\to\infty$ uniformly for all $N>n$ and for all~$x$ in a compact
subset of~$\RR^3$. This establishes the convergence of the series and
completes the proof of the proposition.
\end{proof}

\begin{remark}
Note that each summand $U_{i^l Y_{lm}p}=(2\pi)^{3/2}p(D)[Y_{lm} (\frac{x}{|x|})
 |x|^{-1/2} {J_{\ell + \frac12}(|x|)}]$ of the
series~\eqref{seriesu} is a Beltrami field.
\end{remark}

Since $a_{lm}$ are standard Gaussian variables, it is obvious that the
vector-valued Gaussian field~$u$ has zero mean. Our next goal is to
compute its covariance kernel, $\ka$, which maps each pair of points
$(x,y)\in\RR^3\times\RR^3$ to the symmetric $3\times 3$~matrix
\begin{equation}\label{kappa}
\ka(x,y):=\bE[u(x)\otimes u(y)]\,.
\end{equation}
In particular, we show that this kernel is translationally invariant,
meaning that it only depends on the difference:
\[
  \ka(x,y)=\vka(x-y)\,.
\]
We recall that, by Bochner's theorem, there exists a
nonnegative-definite matrix-valued measure
$\rho$ such that $\vka$ is the Fourier transform of~$\rho$: this is
the spectral measure of the Gaussian random field~$u$. In the
statement, $p_j$ is the $j^{\text{th}}$ component of the vector
field~$p$.

\begin{proposition}\label{P.kernel}
The components of the covariance kernel of the Gaussian random
field~$u$ are
\[
\ka_{jk}(x,y)= \vka_{jk}(x-y)
\]
with
\[
\vka_{jk}(x):= (2\pi)^{\frac32} p_j(D) p_k(-D)\frac{J_{1/2}(|x|)}{|x|^{1/2}}\,.
\]
The spectral measure is $d\rho(\xi)= p(\xi)\otimes
\overline{p(\xi)}\, d\si(\xi)$.
\end{proposition}

\begin{proof}
As $a_{lm}$ are independent standard Gaussian variables,
$\bE(a_{lm}a_{l'm'})=\de_{ll'}\de_{mm'}$, so the covariance matrix
is
\begin{align*}
  \ka_{jk}&(x,y)=  \bE[u_j(x) u_k(y)]=\bE[u_j(x) \overline{u_k(y)}]\\
  &= \sum_{l=0}^\infty\sum_{m=-l}^l
    \sum_{l'=0}^\infty\sum_{m=-l'}^{l'}
    i^{l-l'}\bE(a_{lm}a_{l'm'})\int_\S\int_\S e^{ix\cdot \xi-iy\cdot
    \eta} p_j(\xi)\, \overline{p_k(\eta)}\,
    Y_{lm}(\xi)\,{Y_{l'm'}(\eta)}\, d\si(\xi)\, d\si(\eta)\\
  &= \sum_{l=0}^\infty\sum_{m=-l}^l \int_\S\int_\S e^{ix\cdot \xi-iy\cdot
    \eta} p_j(\xi)\, \overline{p_k(\eta)}\,
    Y_{lm}(\xi)\,{Y_{lm}(\eta)}\, d\si(\xi)\, d\si(\eta)\,.
\end{align*}
Here we have used that $u$ and the spherical harmonics $Y_{lm}$ are real-valued. Since $Y_{lm}$ is an
orthonormal basis, one has that
\[
\sum_{l=0}^\infty\sum_{m=-l}^l  \int_\S\int_\S
\psi(\xi)\,\phi(\eta)\, Y_{lm}(\xi)\, Y_{lm}(\eta)\,
d\si(\xi)\, d\si(\eta)=\int_\S
\psi(\xi)\,\phi(\xi)\, d\si(\xi)
\]
for any functions $\psi,\phi\in L^2(\S)$. Hence we can get rid of the sums in the above formula and write
\begin{align}\label{formulakappa}
  \ka_{jk}(x,y)& = \int_\S e^{i(x-y)\cdot\xi}\,p_j(\xi)\, \overline{p_k(\xi)}
                 \, d\si(\xi)\,,
\end{align}
which yields the formula for the spectral measure of~$u$.
Using now that $p$ is Hermitian (i.e., $\overline{ p(\xi)}=p(-\xi)$)
and a well-known representation formula for the Bessel function
$J_{1/2}$, the above integral can be equivalently written as
\begin{align*}
  \int_\S e^{ix\cdot\xi}\, p_j(\xi)\, \overline{p_k(\xi)}\, d\si(\xi)&= p_j(D)\, {p_k(-D)}\int_\S e^{ix\cdot\xi} d\si(\xi)\\
&=  (2\pi)^{\frac32}p_j(D)\, {p_k(-D)}\frac{J_{1/2}(|x|)}{|x|^{1/2}}\,.
\end{align*}
The proposition then follows.
\end{proof}

A straightforward corollary is that the Gaussian random Beltrami
field~$u$ is normalized so that its covariance matrix is the identity
on the diagonal:

\begin{corollary}\label{C.diag}
  For any~$x\in\RR^3$, $\ka(x,x)=I$.
\end{corollary}

\begin{proof}
The formula for the spectral measure computed in
Proposition~\ref{P.kernel} implies that
\[
\ka_{jk}(x,x)=\int_\S  p_j(\xi)\, \overline{p_k(\xi)}\, d\si(\xi)\,.
\]
As $p$ is a polynomial, the computation then boils down to evaluating integrals of the form
$\int_\S \xi^\al\, d\si(\xi)$, where $\al=(\al_1,\al_2,\al_3)$ is a
multiindex and
$\xi^\al:=\xi_1^{\al_1}\xi_2^{\al_2}\xi_3^{\al_3}$. These integrals
can be computed in closed form~\cite{Fol01}:
\begin{equation}\label{Folland}
  \int_\S \xi^\al\, d\si(\xi)= \begin{cases}
2\big[\prod_{j=1}^3\Gamma(\frac{\al_j+1}2)\big] /\Gamma(\frac{|\al|+3}2) & \text{if $\al_1,\al_2,\al_3$
  are even,}\\[1mm]
0 & \text{otherwise.}
  \end{cases}
\end{equation}
Here $\Ga$ denotes the Gamma function.

Armed with this formula and taking into account the explicit
expression of the polynomial $p(\xi)$ (cf.~Equation~\eqref{p(xi)}),
a tedious but straightforward computation shows
\[
\int_\S  p_j(\xi)\, \overline{p_k(\xi)}\, d\si(\xi)=\de_{jk}\,.
\]
The result then follows.
\end{proof}

\begin{remark}\label{R.rho}
The probability density function of the Gaussian
random vector~$u(x)$ is therefore $\rho(y):= (2\pi)^{-\frac32}\,
e^{-\frac12|y|^2}$. That is, $\bP\{u(x)\in\Omega\}=
\int_\Omega\rho(y)\, dy $ for any $x\in\RR^3$ and any Borel subset~$\Omega\subset\RR^3$.
\end{remark}

Since the Gaussian field~$u$ is of class~$C^\infty$ with
probability~1 by Proposition~\ref{P.conv}, it is standard that it defines a Gaussian probability
measure, which we henceforth denote by~$\bPu$, on the space
of~$C^k$~vector fields on~$\RR^3$, where $k$ is any fixed positive integer.
This space is endowed with its usual Borel $\si$-algebra~$\fS$,
which is the minimal $\si$-algebra containing the ``squares''
\[
  I(x,a,b):=
\{w\in C^k(\RR^3,\RR^3): w(x)\in[a_1,b_1)\times
[a_2,b_2)\times[a_3,b_3)\}
\]
for all $x,a,b,\in\RR^3$. To spell out the details, let us denote
by~$\Omega$ the sample space of the random variables $a_{lm}$ and show
that the random field~$u$ is a measurable map from~$\Omega$ to~$C^k(\RR^3,\RR^3)$. Since the
$\si$-algebra of~$C^k(\RR^3,\RR^3)$ is generated by point evaluations,
it suffices to show that
\[
u(x)= \sum_{l=0}^\infty \sum_{m=-l}^l a_{lm}\, U_{i^l Y_{lm}p}(x)
\]
is a measurable function
$\Omega\to\RR^3$ for each~$x\in\RR^3$. But this is obvious because $u(x)$ is
the limit of finite linear combinations (with coefficients in~$\RR^3$)
of the random variables~$a_{lm}$, which are of course measurable. In what follows, we will not
mention the $\si$-algebra explicitly to keep the notation simple. Also, in view of the later applications to invariant
tori, we will henceforth assume that~$k\geq4$.

Following Nazarov and Sodin~\cite{NS16}, the next proposition shows that
from the facts that the
covariance kernel $\ka(x,y)$ only depends on~$x-y$ and that the
spectral measure has no atoms one can infer two useful properties of
our Gaussian probability measure that
will be extensively employ in the rest of the paper. Before stating
the result, let us recall that the probability
measure~$\bPu$ is said to be {translationally~invariant} if $\bPu(\tau_y \cA)=\bPu(\cA)$
for all~$\cA\subset\fS$ and all~$y\in \RR^3$. Here $\tau_y$~denotes
the translation operator on $C^k$~fields, defined as $\tau_y w(x):=
w(x+y)$.

\begin{proposition}\label{P.ergodic}
The probability measure~$\bPu$ is translationally
invariant. Furthermore, if~$\Phi$ is an $L^1$~random variable on the probability space
$(C^k(\RR^3,\RR^3),\fS,\bPu)$ , then
\[
\lim_{R\to\infty}\Bint_{B_R} \Phi\circ\tau_y\, dy = \bE\Phi
\]
both $\bPu$-almost surely and in $L^1(C^k(\RR^3,\RR^3),\bPu)$.
\end{proposition}

\begin{proof}
  Since the covariance kernel $\ka(x,y)$ only depends on~$x-y$, the
  probability measure~$\bPu$ is translationally invariant. Also, note
  that $(y,w)\mapsto \tau_yw$ defines a continuous map
  \[
\RR^3\times C^k(\RR^3, \RR^3)\to C^k(\RR^3, \RR^3)\,,
  \]
  so the map $(y,w)\mapsto \Phi(\tau_yw)$ is measurable on the product
  space $\RR^3 \times C^k(\RR^3, \RR^3)$. Wiener's ergodic
theorem~\cite{NS16, Bec81} then ensures that, for $\Phi$ as in the
statement, there is a random variable $\Phi^*\in
L^1(C^k(\RR^3\times\RR^3),\bPu)$ such that
\[
\Bint_{B_R} \Phi\circ\tau_y\, dy \Lone \Phi^*
\]
as $R\to\infty$. Furthermore, $\Phi^*$ is translationally invariant
(i.e., $\Phi^*\circ\tau_y=\Phi^*$ for all $y\in\RR^3$ almost surely) and $\bE\Phi^* =
\bE\Phi$.

Also, as the spectral measure (computed in Proposition~\ref{P.kernel}
above) has no atoms, a theorem of Grenander, Fomin and Maruyama (see
e.g.~\cite[Appendix B]{NS16} or~\cite{Grenander} and note that the
proof carries over to the multivariate and vector-valued case) ensures that the action of the translations
$\{\tau_y: y\in\RR^3\}$ on the probability space
$(C^k(\RR^3,\RR^3),\fS,\bPu)$ is ergodic. As the measurable
function~$\Phi^*$ is translationally invariant, one then infers that
$\Phi^*$ is constant $\bPu$-almost surely. As $\Phi$ and $\Phi^*$ have
the same expectation, then $\Phi^*=\bE \Phi$ almost surely.
The proposition then follows.
\end{proof}

It is clear that the support of the probability measure~$\bPu$ must
be contained in the space of Beltrami fields. In the last result of this section,
we show that the support is in fact the whole space. This property will be key in the following sections.

\begin{proposition}\label{P.support}
The support of the Gaussian probability measure $\mu_u$ is the space of Beltrami fields. More precisely, let $v$ be a Beltrami field. For any compact set $K\subset\RR^3$ and
each~$\ep>0$,
\[
\bPu\big(\big\{ w\in C^k(\RR^3,\RR^3): \|v-w\|_{C^k(K)}<\ep\big\}\big)>0\,.
\]
\end{proposition}

\begin{proof}
By Proposition~\ref{P.approx}, there exists a Hermitian finite linear
combination of spherical harmonics,
\[
\vp=\sum_{l=0}^n\sum_{m=-l}^l i^l \al_{lm} Y_{lm}\,,
\]
where $\al_{lm}$ are
real numbers (not random variables), such that $\|v-U_{\vp p}\|_{C^k(K)}<\ep/4$. Hence
\[
\bPu\big (\big\{ w\in C^k(\RR^3,\RR^3):
    \|w-v\|_{C^k(K)}<\ep\big\}\big) \geq \bP\bigg (\bigg\{\|u-U_{\vp p}\|_{C^k(K)}<\frac\ep4\bigg\}\bigg)\,,
\]
where $\bP$ denotes the natural Gaussian probability measure on the space of
sequences $(a_{lm})$.

Proposition~\ref{P.conv} shows that the series
\[
  \sum_{l=0}^\infty\sum_{m=-l}^l a_{lm} U_{i^lY_{lm}p}
\]
converges in~$C^k(K)$ almost surely, so for any fixed $\de>0$ there
exists some number~$N$ (which one can assume larger than~$n$) such that
\[
\bP\bigg(\bigg\{\bigg\| \sum_{l=N+1}^\infty\sum_{m=-l}^l a_{lm} U_{i^l
        Y_{lm}p}\bigg\|_{C^k(K)}<\frac\ep8\bigg\}\bigg)>1-\de\,.
\]
With the convention that $\al_{lm}:=0$
for $l>n$, note that
\[
\|u-U_{\vp p}\|_{C^k(K)}\leq \sum_{l=0}^N\sum_{m=-l}^l
|a_{lm}-\al_{lm}| \|U_{i^lY_{lm}p}\|_{C^k(K)} + \left\| \sum_{l=N+1}^\infty\sum_{m=-l}^l a_{lm} U_{i^lY_{lm}p}\right\|_{C^k(K)}\,.
\]
Therefore, if we set $M:= 8(N+1)^2\max_{l\leq N}\max_{-l\leq m\leq l}
\|U_{i^lY_{lm}p}\|_{C^k(K)}$, it follows that
\begin{multline*}
  \bP\left(\left\{\|u-U_{\vp p}\|_{C^k(K)}<\frac\ep4\right\}\right) \\
  \geq
 \bP\bigg (\bigg\{\bigg\| \sum_{l=N+1}^\infty\sum_{m=-l}^l a_{lm}
       U_{i^lY_{lm}p}\bigg\|_{C^k(K)}<\frac\ep8
  \bigg\}\bigg)\, \prod_{l=0}^N \prod_{m=-l}^l \bP\left(\left\{
    |a_{lm}-\al_{lm}|<\frac\ep M\right\}\right)\,,
\end{multline*}
which is strictly positive. The proposition then follows.
\end{proof}

\section{Preliminaries about hyperbolic periodic orbits and invariant tori}
\label{S.preliminaries}

In this section we construct Beltrami fields that exhibit hyperbolic periodic orbits or a positive measure set of ergodic invariant tori of arbitrary topology. Our constructions are robust in the sense that these properties hold for any other divergence-free field that is $C^4$-close to the Beltrami field. Additionally, we recall some basic notions and results about periodic orbits and invariant tori that will be useful in the following sections.

\subsection{Hyperbolic periodic orbits}\label{per}

We recall that a periodic integral curve, or periodic orbit, $\gamma$ of a vector field $u$ is hyperbolic if all the (possibly complex) eigenvalues $\la_j$ of the monodromy matrix of $u$ at $\ga$ have modulus $|\la_j|\neq 1$. Since we are interested in divergence-free vector fields in dimension 3, in this case the eigenvalues are of the form $\la,\la^{-1}$ for some real $\la>1$. The maximal Lyapunov exponent of the periodic orbit $\gamma$ is defined as $\La:=\frac{\log \la}{T}>0$, where $T$ is the period of $\ga$.

Given a closed curve $\ga_0$ smoothly embedded in $\RR^3$, we say that
$\ga$ has the knot type $[\ga_0]$ if $\ga$ is isotopic to $\ga_0$. It
is well known that the number of knot types is countable. Given a set
of four positive numbers $\cI=(T_1,T_2,\La_1,\La_2)$, with $0<T_1<T_2$
and~$0<\La_1<\La_2$, we denote by $\No(R;[\ga],\cI)$ the number of
hyperbolic periodic orbits of a vector field $u$ contained in the ball
$B_R$, of knot type $[\ga]$, whose periods and maximal Lyapunov
exponents are in the intervals $(T_1,T_2)$ and $(\La_1,\La_2)$,
respectively. Since we have fixed the intervals of the periods and
Lyapunov exponents, there is a neighborhood of thickness $\eta_0$ of
each periodic orbit ($\eta_0$ independent of the orbit) such that no
other periodic orbit of this type intersects it. The compactness of~$B_R$ then immediately implies that $\No(R,[\ga],\cI)$ is finite,
although the total number of hyperbolic periodic orbits in~$B_R$ may
be countable.

An easy application of the hyperbolic permanence theorem~\cite[Theorem
1.1]{HPS} implies that the above periodic orbits are robust under
$C^1$-small perturbations, so that
\[
  N_v^{\mathrm{o}}(R;[\ga],\cI)\geq
  \No(R;[\ga],\cI)
\]
for any vector field~$v$ that is close enough to~$u$
in the $C^1$~norm. Indeed, if $\|u-v\|_{C^1(B_R)}<\de$, then $v$ has a
periodic orbit $\ga_\de$ that is isotopic to, and contained in
a tubular neighborhood of width~$C\de$ of, each periodic orbit~$\ga$
of~$u$ that has the aforementioned properties. Moreover, the period
and maximal Lyapunov exponent of $\ga_\de$ is also $\de$-close to that
of $\ga$, so choosing $\de$ small enough they still lie in the
intervals $(T_1,T_2)$ and $(\La_1,\La_2)$, respectively. Thus we have
proved the following:

\begin{proposition}\label{P.lscpo}
The functional $u\mapsto \No(R;[\ga],\cI)$ is
lower semicontinuous in the $C^k$~compact open topology for vector
fields, for any $k\geq 1$. Furthermore, $\No(R;[\ga],\cI)<\infty$ for
any $C^1$~vector field~$u$.
\end{proposition}

The following result ensures that, for any fixed knot type~$[\ga]$ and
any quadruple~$\cI$, there is a Beltrami field~$u$ for which
$\No(R;[\ga],\cI)\geq 1$. This result is a consequence
of~\cite[Theorem 1.1]{Annals}, so we just
give a short sketch of the proof.

\begin{proposition}\label{P.perBelt}
  Given a closed curve $\ga_0\subset \RR^3$ and a set of numbers $\cI$
  as above, there exists a Hermitian finite linear combination of
  spherical harmonics $\varphi$ such that the Beltrami field
  $u_0:=U_{\vp p}$ has a hyperbolic periodic orbit $\ga$ isotopic to
  $\ga_0$, whose period and maximal Lyapunov exponent lie in the
  intervals $(T_1,T_2)$ and $(\La_1,\La_2)$, respectively.
\end{proposition}
\begin{proof}
Proceeding as in~\cite[Section 3, Step 2]{Annals}, after perturbing slightly the curve $\ga_0$ to make it real analytic (let us also call $\ga_0$ the new curve), we construct a narrow strip $\Sigma$ that contains the curve $\ga_0$. Using the same coordinates $(z,\theta)$ as introduced in~\cite[Section 5]{Annals}, we define an analytic vector field
$$w:=\frac{|\ga_0|}{T}\,\nabla \theta-\La\,z\nabla z\,,$$
where $|\ga_0|$ is the length of $\ga_0$ and $T\in (T_1,T_2)$, $\La\in (\La_1,\La_2)$. Using the Cauchy--Kovalevskaya theorem for Beltrami fields~\cite[Theorem 3.1]{Annals}, we obtain a Beltrami field $v$ on a neighborhood of $\ga_0$ such that $v|_{\Sigma}=w$. A straightforward computation shows that $\ga_0$ is a hyperbolic periodic orbit of $v$ of period $T$ and maximal Lyapunov exponent $\La$. The result immediately follows by applying Proposition~\ref{P.approx}.
\end{proof}
\begin{corollary}\label{C.perBelt}
There exists $R_0>0$ and $\delta>0$ such that $N^{\mathrm{o}}_w(R_0;[\ga],\cI)\geq 1$ for any vector field $w$ such that $\|w-u_0\|_{C^k(B_{R_0})}<\delta$, provided that $k\geq 1$.
\end{corollary}
\begin{proof}
Taking $R_0$ large enough so that the periodic orbit $\ga$ is contained in $B_{R_0}$, the result is a straightforward consequence of the lower semicontinuity of $N^{\mathrm{o}}_u(R;[\ga],\cI)$, cf. Proposition~\ref{P.lscpo}.
\end{proof}

\subsection{Nondegenerate invariant tori}\label{tor}

We recall that an invariant torus $\cT$ of a vector field~$u$ is a
compact surface diffeomorphic to the 2-torus, smoothly embedded
in~$\RR^3$, and such that, the field~$u$ is tangent to~$\cT$ and does
not vanish on~$\cT$. In other words, $\cT$~is invariant under the flow of~$u$. Given an embedded torus~$\cT_0$, we say that~$\cT$ has the knot type~$[\cT_0]$ if~$\cT$ is isotopic to~$\cT_0$. It is well known that the number of knot types of embedded tori is countable.

To study the robustness of the invariant tori of a vector field it is customary to introduce two concepts: an arithmetic condition (called Diophantine), which is related to the dynamics of~$u$ on~$\cT$, and a nondegeneracy condition (called twist) that is related to the dynamics of~$u$ in the normal direction to~$\cT$.

We say that the invariant torus~$\cT$ is Diophantine with Diophantine
frequency~$\om$ if there exist global coordinates on the torus $(\theta_1,\theta_2)\in
(\RR/\mathbb Z)^2$ such that the restriction of the field~$u$
to~$\cT$ reads in these coordinates as
\begin{equation}\label{eq.torus}
u|_\cT=a\, e_{\te_1}+b\, e_{\te_2}\,,
\end{equation}
for some nonzero real constants $a,b$, and $\om:=a/b$ modulo~$1$ is a
Diophantine number. This means that there exist constants~$c>0$ and~$\nu>1$ such that
\[
\Big|\om-\frac{p}{m}\Big|\geq \frac{c}{m^{\nu+1}}
\]
for any integers $p,m$ with $m\geq 1$. Here $e_{\te_j}$ (often denoted
by~$\pd_{\te_j}$) denotes the tangent
vector in the direction of~$\te_j$. We recall that the set of
Diophantine numbers (with all $c>0$ and all $\nu>1$) has full
measure. The value of the frequency~$\omega$, modulo~$1$, is
independent of the choice of coordinates.

Let us now introduce the notion of twist, which is more involved. To this end, we parameterize a neighborhood of $\cT$ with a coordinate system $(\rho,\theta_1,\theta_2)\in (-\delta,\delta)\times (\RR/\mathbb Z)^2$ such that $\cT=\{\rho=0\}$ and $u|_{\rho=0}$ has the form~\eqref{eq.torus}. Let us now compute the Poincar\'e map $\pi$ defined by the flow of $u$ on a transverse section $\Sigma\subset \{\theta_2=0\}$ (which exists if $\delta$ is small enough because $b\neq 0$):
\begin{align}
\pi:(-\delta',\delta')\times (\RR/\mathbb Z) &\to
  (-\delta,\delta)\times (\RR/\mathbb Z)\\
  (\rho,\theta_1) &\mapsto (\pi_1(\rho,\theta_1),\pi_2(\rho,\theta_1))\,,
\end{align}
for $\de'<\de$. Obviously,
$\pi(0,\theta_1)=(0,\theta_1+\omega)$. Since $u$ is divergence-free,
the map $\pi$ preserves an area form $\sigma$ on $\Sigma$, which one
can write in these coordinates as \begin{equation}
\sigma=F(\rho,\theta_1)\,d\rho\wedge d\theta_1\,,
\end{equation}
for some positive function $F$. Notice that the area form $\sigma$ is
exact because it can be written as $\sigma=dA$, where~$A$ is the 1-form
\[
 A:=h(\rho,\theta_1)\,d\theta_1\,,\qquad h(\rho,\theta_1):=\int_{-\delta}^{\rho}F(s,\theta_1)\,ds\,,
\]
and the map $\pi$ is also exact in the sense that $\pi^*A-A$ is an exact $1$-form. Indeed, the area preservation implies that $d(\pi^*A-A)=0$; moreover the periodicity of $h$ in $\theta_1$ readily implies that
\[
\int_0^1(\pi^*A-A)|_{\rho=0}=\int_0^1(h(0,\theta_1+\omega)-h(0,\theta_1))\,d\theta_1=0\,,
\]
so the claim follows from De Rham's theorem. The exactness of both $\sigma$ and $\pi$ is a crucial ingredient to apply the KAM theory.

\begin{remark}
It was shown in~\cite[Proposition 7.3]{Acta} that if the Euclidean volume form $dx$ reads as $H(\rho,\theta_1,\theta_2)\,d\rho\wedge d\theta_1\wedge d\theta_2$ in coordinates $(\rho,\theta_1,\theta_2)$ for some positive function~$H$, then the factor~$F$ that defines the area form $\sigma$ is $F(\rho,\te_1)=H(\rho,\theta_1,0)u_{\theta_2}(\rho,\theta_1,0)$, where $u_{\theta_2}$ denotes the $\theta_2$-component of the vector field~$u$.
\end{remark}

The twist of the invariant torus $\cT$ is then defined as the number
\begin{equation}\label{eq.twist}
\tau:=\int_0^1\frac{\partial_\rho \pi_2(0,\theta_1)}{F(0,\theta_1)}\,d\theta_1\,.
\end{equation}
The reason for which we consider this quantity is that it crucially
appears in the KAM nondegeneracy condition of~\cite{GL08}, cf. Ref.~\cite[Definition 7.5]{Acta} for this particular case.

In the present paper we are interested in the volume of the set of
invariant tori of a divergence-free vector field~$u$. More precisely,
given a
quadruple $\cJ:=(\omega_1,\omega_2, \tau_1,\tau_2)$, where
$0<\om_1<\om_2$, $0<\tau_1<\tau_2$, we denote by $\Vt(R;[\cT],\cJ)$
the inner measure of the set of Diophantine invariant tori of a
vector field~$u$ contained in the ball~$B_R$, of knot type~$[\cT]$,
whose frequencies and twists are in the intervals~$(\om_1,\om_2)$
and~$(\tau_1,\tau_2)$, respectively. One must employ the inner
measure of this set (as opposed to its usual volume) because this set does not need to be measurable.
When we speak of the volume of
this set, it should always be understood in this sense. An efficient
way of providing a lower bound for this volume is by considering, for each~$V_0>0$, the
number $\Nt(R;[\cT],\cJ,V_0)$ of pairwise disjoint (closed) solid tori contained in~$B_R$ whose
boundaries are Diophantine invariant tori with parameters in~$\cJ$ and which
contain a set of Diophantine invariant tori with parameters in~$\cJ$ of inner
measure greater that~$V_0$.

\begin{remark}
The twist defined in Equation~\eqref{eq.twist} depends on several
choices we made to construct the Poincar\'e map (i.e., the transverse
section and the coordinate system). Accordingly, the functional
$\Vt(R;[\cT],\cJ)$ has to be understood as the inner measure of the
set of Diophantine invariant tori whose twists lie in the interval
$(\tau_1,\tau_2)$ for some choice of  (suitably bounded) coordinates
and sections, and similarly with~$\Nt(R;[\cT],\cJ,V_0)$. It is well known that the property of nonzero twist is independent of the aforementioned choices.
\end{remark}

Since the Poincar\'e map $\pi$ that we introduced above is exact, we
can apply the KAM theorem for divergence-free vector
fields~\cite[Theorem 3.2]{KKP} to show that the above invariant tori
are robust for $C^4$-small perturbations, so that $V_v^{\mathrm
  t}(R;[\cT],\cJ)\geq \Vt(R;[\cT],\cJ)+o(1)$ and $N_v^{\mathrm
  t}(R;[\cT],\cJ,V_0)\geq \Nt(R;[\cT],\cJ,V_0)$ for any divergence-free
vector field~$v$ that is $C^4$-close to~$u$. Indeed, if
$\|u-v\|_{C^4(B_R)}<\de$, then~$v$ has a set of Diophantine invariant tori of knot
type~$[\cT]$ and of volume
$$
V_v^{\mathrm t}(R;[\cT],\cJ)\geq \Vt (R;[\cT],\cJ)-C\de^{1/2}\,.
$$
Here we have used that the frequency and twist of each of these invariant tori is
$\de$-close to those of~$u$, so by choosing~$\de$ small enough they
lie in the intervals $(\om_1,\om_2)$ and $(\tau_1,\tau_2)$,
respectively. The argument for~$\Nt(R;[\cT],\cJ,V_0)$ is analogous. Summing up, we have proved the following:

\begin{proposition}\label{P.lscit}
The functionals $u\mapsto \Nt(R;[\cT],\cJ,V_0)$ and $u\mapsto
\Vt(R;[\cT],\cJ)$ are lower semicontinuous
in the $C^k$~compact open topology for divergence-free vector fields,
for any $k\geq 4$.
\end{proposition}

We next show that, for any knot type~$[\cT]$, one can pick a
quadruple~$\cJ$ and some $V_0>0$ for which there is a Beltrami
field~$u$ with $\Nt (R;[\cT],\cJ,V_0)\geq 1$. This is a straightforward
consequence of~\cite[Theorem 1.1]{Acta} (see also~\cite[Section
3]{Alex}), so we just sketch the proof.

\begin{proposition}\label{P.toriBelt}
Given an embedded torus $\cT\subset \RR^3$, there exists a set of
numbers $\cJ, V_0$ as above, and a Hermitian finite linear combination of
spherical harmonics $\vp$ such that the Beltrami field $u_0:=U_{\vp p}$ has
a set of inner measure greater than~$V_0>0$ that consists of Diophantine invariant
tori of knot type~$[\cT]$ whose frequencies and twists lie in the
intervals $(\om_1,\om_2)$ and $(\tau_1,\tau_2)$, respectively.
\end{proposition}
\begin{proof}
It follows from~\cite[Theorem 1.1]{Acta} that there exists a Beltrami field $v$ that satisfies $\curl v=\lambda v$ in $\RR^3$ for some small constant $\lambda>0$, which has a positive measure set of invariant tori of knot type~$[\cT]$. These tori are Diophantine and have positive twist. It is obvious that the field $u(x):=v(x/\lambda)$ satisfies the equation $\curl u=u$ in $\RR^3$, and still has a set of Diophantine invariant tori of knot type $[\cT]$ of measure bigger than some constant $V_0$, and positive twist. The result follows taking the intervals $(\om_1,\om_2)$ and $(\tau_1,\tau_2)$ in the definition of $\cJ$, so that they contain the frequencies and twists of these tori of~$u$, and applying Proposition~\ref{P.approx} to approximate~$u$ by a Beltrami field $U_{\vp p}$ in a large ball containing the aforementioned set of invariant tori.
\end{proof}
\begin{corollary}\label{C.torBelt}
Take $\cJ$ and $V_0$ as in Proposition~\ref{P.toriBelt}. There exists $R_0>0$ and $\delta>0$ such that $N^{\mathrm{t}}_w(R_0;[\cT],\cJ,V_0)\geq 1$ and $V^{\mathrm{t}}_w(R_0;[\cT],\cJ)>V_0/2$ for any divergence-free vector field $w$ such that $\|w-u_0\|_{C^k(B_{R_0})}<\delta$, provided that $k\geq 4$.
\end{corollary}
\begin{proof}
Taking $R_0$ large enough so that the aforementioned set of invariant tori of $u_0$ is contained in $B_{R_0}$, the result is a straightforward consequence of the lower semicontinuity of $N^{\mathrm{t}}_u(R;[\cT],\cJ,V_0)$ and $V^{\mathrm{t}}_u(R;[\cT],\cJ)$, cf. Proposition~\ref{P.lscit}.
\end{proof}

\section{A Beltrami field on~$\RR^3$ that is stably chaotic}
\label{S.chaos}

Our objective in this section is to construct a Beltrami field $u$ in
$\RR^3$ that exhibits a horseshoe, that is, a compact (normally) hyperbolic
invariant set on which the time-$T$ flow of $u$ (or of a suitable
reparametrization thereof) is topologically conjugate to a Bernoulli
shift. It is standard that a horseshoe of a three-dimensional flow is a connected branched surface, and that the existence of a horseshoe is stable in the
sense that any other field that is $C^1$-close to $u$ has a horseshoe too~\cite[Theorem 5.1.2]{GH}. Moreover, the existence of a horseshoe implies that the field
has positive topological entropy; recall that the topological
entropy of the field, which we denote as $\htop(u)$, is defined as the entropy of its time-$1$
flow. Summarizing, we have the following result for the number of
(pairwise disjoint) horseshoes of $u$ contained in $B_R$, $N^{\mathrm{h}}_u(R)$:

\begin{proposition}\label{P.lsch}
The functional $u\mapsto N^{\mathrm{h}}_u(R)$ is
lower semicontinuous in the $C^k$~compact open topology for vector
fields, for any $k\geq 1$. Moreover, if $u$ has a horseshoe, its topological entropy is positive.
\end{proposition}

In short, the basic idea to construct a Beltrami field with a horseshoe, is to construct first ``an integrable'' Beltrami field
having a heteroclinic cycle between two hyperbolic periodic orbits,
which we subsequently perturb within the Beltrami class to produce a
transverse heteroclinic intersection. By the Birkhoff--Smale theorem,
this ensures the existence of horseshoe-type dynamics.

\begin{proposition}\label{P:chaos}
There exists a Hermitian finite linear combination of spherical
harmonics~$\vp$ such that the Beltrami field $u_0:=U_{\vp p}$ exhibits a horseshoe. In other words, $N^{\mathrm{h}}_{u_0}(R_0)\geq 1$ for all large enough $R_0>0$.
\end{proposition}

\begin{proof}
Let us take cylindrical coordinates $(z,r,\te)\in \RR\times
\RR^+\times \TT$, with $\TT:=\RR/2\pi\mathbb Z$, defined as
\[
z:= x_3 \,,\qquad (r\cos\te,r\sin\te):=(x_1,x_2)\,.
\]
We now consider the axisymmetric vector field $v$ in $\RR^3$ given by
\begin{equation}\label{eq_axi}
v:=\frac1r\Big(\partial_r\psi\,
E_z-\partial_z\psi\,E_r+\frac{\psi}{r}\,E_\te\Big)\,.
\end{equation}
Here
\[
\psi:=\cos z+3rJ_1(r)
\]
with $J_1$ being the Bessel function of the first kind and order~1,
and the vector fields
\[
E_z:= (0,0,1)\,,\qquad E_r:= \frac1r(x_1,x_2,0)\,,\qquad E_\te:= (-x_2,x_1,0)\,,
\]
which are often denoted by $\pd_z$, $\pd_r$, $\pd_\te$ in the
dynamical systems literature, have been chosen so that
\[
E_z\cdot\nabla\phi=\pd_z\phi\,,\qquad
E_r\cdot\nabla\phi=\pd_r\phi\,,\qquad E_\te\cdot\nabla\phi=\pd_\te\phi
\]
for any function~$\phi$. Notice that $v\cdot \nabla\psi=0$, so the
scalar function~$\psi$ is a first integral of $v$. This means that the
trajectories of the field~$v$ are tangent to the level sets of $\psi$.

The vector field $v$ is not defined on the $z$-axis, so we shall
consider the domain in Euclidean 3-space
$$\Omega:=\left\{(z,r,\te):\,
  (z,r)\in\cD\,,\; \te\in\TT\right\}\,,$$
where $\cD$ is the domain in the $(z,r)$-plane given by
$$\cD:=\left\{(z,r):\,-10<z<10,\,
\frac{9}{10}<r<\frac{18}{5}\right\}\,.$$
The reason for choosing this particular domain of
$\RR^3$ will become clear later in the proof; for the time being, let
us just note that $\psi(z,r)>0$ if $(z,r)\in\cD$.

Also, observe that, away
from the axis $r=0$, the vector field~$v$ is smooth and satisfies the
Beltrami field equation $\curl v=v$.

We claim that, in $\Omega$, $v$ has two hyperbolic periodic orbits joined by a heteroclinic cycle. Indeed, noticing that
\[
(\partial_z\psi,\partial_r\psi)=(-\sin z,3rJ_0(r))\,,
\]
where we have used the identity $\pd_r[rJ_1(r)]=rJ_0(r)$, it follows that
the points $p_{\pm}:=(\pm \pi,j_{0,1})\in \cD$ are critical points
of~$\psi$. Here $j_{0,1}=2.4048\dots$ is the first zero of the Bessel
function~$J_0$. Plugging this fact in Equation~\eqref{eq_axi}, this
implies that, on the circles in 3-space
\[
\ga_\pm:= \{(z,r,\te): (z,r)=p_\pm\,,\; \te\in\TT\}\,,
\]
the field~$v$ takes the form
\[
v(p_\pm,\te)=\frac{c_0}{j_{0,1}^2}\,E_\te
\]
with $c_0:=3j_{0,1}J_1(j_{0,1})-1>0$. Therefore, we conclude that the circles
$\gamma_\pm$ are periodic orbits of~$v$ contained in~$\Omega$.

It is standard that the stability of these periodic orbits can be
analyzed using the associated normal variational equation. Denoting
by $(v_z,v_r,v_\te)$ the components of the field~$v$ in the basis
$\{E_z,E_r,E_\te\}$, this is the linear ODE
\[
\dot\eta = A\eta\,,
\]
where $\eta$ takes values in~$\RR^2$ and~$A$ is the
constant matrix
\[
A:= \frac{\pd(v_z,v_r)}{\pd(z,r)}\bigg|_{(z,r)=p_\pm}= \left(
  \begin{array}{cc}
    0 & 3J'_0(j_{0,1}) \\
    -{1}/{j_{0,1}} &0 \\
  \end{array}
\right)\,.
\]
The Lyapunov exponents of the periodic orbit~$\ga_\pm$ are the
eigenvalues of the matrix~$A$. Therefore, since $J'_0(j_{0,1})<0$, these
periodic orbits have a positive and a negative Lyapunov exponent, so
they are hyperbolic periodic orbits of saddle type.

Since $\psi$ is a first integral of $v$ and $\psi(p_\pm)=c_0$, the set
\[
\{(z,r,\te): \psi(z,r)=c_0\}
\]
is an invariant singular surface of the vector field $v$. This set
contains two regular surfaces $\Gamma_1$ and $\Gamma_2$ diffeomorphic
to a cylinder. We label them so $\Gamma_1$ is contained in the half
space~$\{r\leq j_{0,1}\}$ and~$\Gamma_2$ in~$\{r\geq j_{0,1}\}$. The
boundaries of these cylinders are the periodic
orbits~$\gamma_\pm$. The surface~$\Gamma_1$ is the stable manifold
of~$\gamma_+$ that coincides with an unstable manifold of~$\gamma_-$,
while $\Gamma_2$ is the unstable manifold of~$\gamma_+$ that coincides
with a stable manifold of~$\gamma_-$. Hence the union $\Gamma_1\cup
\Gamma_2$ of both cylinders then form an heteroclinic cycle between
the periodic orbits~$\ga_+$ and~$\ga_-$, and one can see that it is contained in~$\Omega$.

Let us now perturb the Beltrami field $v$ in $\Omega$ by adding a
vector field $w$ (to be fixed later) that also satisfies the Beltrami
field equation $\curl w=w$. Our goal is to break the heteroclinic
cycle $\Gamma_1\cup \Gamma_2$ in order to produce transverse
intersections of the stable and unstable manifolds of
$\gamma_+^\varepsilon$ and $\gamma_-^\varepsilon$, where
$\gamma_{\pm}^\varepsilon$ denote the hyperbolic periodic
orbits of the perturbed vector field
\[
X:=v+\varepsilon w=\left(\frac{\partial_r\psi}{r}+\varepsilon w_z\right)\, E_z+\left (-\frac{\partial_z\psi}{r}+\varepsilon w_r\right)\,E_r+\left (\frac{\psi}{r^2}+\varepsilon w_\te\right)\,E_\te\,.
\]
As before, $(w_z,w_r,w_\te)$ denote the components of the vector field $w$ in
the basis $\{E_z,E_r,E_\te\}$, which are functions of all three cylindrical coordinates~$(z,r,\te)$. If $\varepsilon>0$ is small enough, the
$\te$-component of $X$ is positive on the domain $\Omega$, so we can
divide $X$ by the factor $X_\te:=\frac{\psi}{r^2}+\varepsilon w_\te>0$
to obtain another vector field $Y$ that has the same integral curves
up to a reparametrization:
\begin{equation}\label{eq_Y}
Y:=\frac{X}{X_\te}=\frac{r\partial_r\psi+\varepsilon r^2w_z}{\psi+\varepsilon r^2w_\te}\, E_z+\frac{-r\partial_z\psi+\varepsilon r^2w_r}{\psi+\varepsilon r^2w_\te}\,E_r+E_\te\,.
\end{equation}

Substituting the expression of $\psi(z,r)$ and expanding in the small
parameter $\varepsilon$, the analysis of the integral curves of $Y$
reduces to that of the following non-autonomous system of ODEs in the
planar domain $\cD$:
\begin{align}\label{eq_pert}
&\frac{d z}{dt}=\frac{3r^2J_0(r)}{\psi(z,r)}+\varepsilon \left(\frac{r^2w_z(z,r,t)}{\psi(z,r)}-\frac{3r^4J_0(r)w_\te(z,r,t)}{\psi(z,r)^2}\right)+O(\varepsilon^2)\,,\\\label{eq_pert2}
&\frac{d r}{dt}=\frac{r\sin z}{\psi(z,r)}+\varepsilon\left
                                                                                                                                                                         (\frac{r^2w_r(z,r,t)}{\psi(z,r)}-\frac{r^3\sin z \, w_\te(z,r,t)}{\psi(z,r)^2}\right)+O(\varepsilon^2)\,.
\end{align}
Notice that the dependence on $t$ is $2\pi$-periodic, and that we have
replaced~$\te$ by~$t$ in the function $w_z(z,r,\te)$ (and
similarly~$w_r,w_\te$) because the $\te$-component of the vector
field~$Y$ is~1. When
$\varepsilon=0$, one has
\begin{align}\label{eq_unpert}
&\dot z=\frac{3r^2J_0(r)}{\psi(z,r)}\,,\\\label{eq_unpert2}
&\dot r=\frac{r\sin z}{\psi(z,r)}\,.
\end{align}
Hence the unperturbed system is Hamiltonian with symplectic form
$\omega:=r^{-1}dz\wedge dr$ and Hamiltonian
function $H(z,r):=\log \psi(z,r)$. The periodic orbits $\gamma_\pm$ of
$v$ and their heteroclinic cycle $\Gamma_1\cup\Gamma_2$ correspond to
the (hyperbolic) fixed points $p_\pm$ of the unperturbed system joined
by two heteroclinic connections $\widetilde \Gamma_k:=\Gamma_k\cap
\{\te=0\}$, $k=1,2$. These are precisely the two pieces of the level
curve $\{H(z,r)=\log c_0\}$ that are contained in~$\cD$. Let us denote
by
$$\gamma_k(t)=(Z_k(t;0,r_k),R_k(t;0,r_k))$$
the integral
curves of the separatrices that solve Equations~\eqref{eq_unpert}-\eqref{eq_unpert2} with
initial conditions $(0,r_k)\in \widetilde \Gamma_k$. Of course, the
closure of the set ${\{\gamma_k(t):t\in\RR\}}$ is $\widetilde
\Gamma_k$, and the stability analysis of the periodic integral curves~$\gamma_\pm$
readily implies that $\lim_{t\to\pm(-1)^{k+1}\infty}\gamma_k(t)=p_\pm$.

By the implicit function theorem, the perturbed system~\eqref{eq_pert}-\eqref{eq_pert2}
has exactly two hyperbolic fixed points $p_{\pm}^\varepsilon\in \cD$ so that $p_{\pm}^\varepsilon \to p_\pm$ as $\varepsilon\to 0$. The technical tool to prove that the unstable (resp. stable) manifold of $p_+^\varepsilon$ and the stable (resp. unstable) manifold of $p_-^\varepsilon$ intersect transversely when $\varepsilon>0$ is small is the Melnikov function. We define the vector fields $Y_0$, $Y_1$, respectively, as the unperturbed system and the first order in $\varepsilon$ perturbation, i.e.,
\begin{align*}
&Y_0:= \frac{3r^2J_0(r)}{\psi(z,r)}\,E_z+\frac{r\sin z}{\psi(z,r)}\,E_r\,,\\
&Y_1:= \left(\frac{r^2w_z}{\psi(z,r)}-\frac{3r^4J_0(r)w_\te}{\psi(z,r)^2}\right)\,E_z+\left(\frac{r^2w_r}{\psi(z,r)}-\frac{r^3\sin z w_\te}{\psi(z,r)^2}\right)\,E_r\,.
\end{align*}
Since the unperturbed system is Hamiltonian, we can apply
Lemma~\ref{L:melni} below (which is a variation on known results in
Melnikov theory) to conclude that if the Melnikov functions
\begin{equation}\label{eq:melni}
M_k(t_0):=\int_{-\infty}^\infty \om(Y_0,Y_1)|_{\gamma_k(t-t_0)}\, dt\,,
\end{equation}
have simple zeros for each $k=1,2$, then the aforementioned transverse
intersections exist, and that actually the heteroclinic connections
intersect at infinitely many points. The integrand $\om(Y_0,Y_1)$
denotes the action of the symplectic 2-form~$\om$ on the vector fields
$Y_0,Y_1$, evaluated on the integral curve
$\gamma_k(t-t_0)$. It is standard that the improper integral
in the definition of the Melnikov functions is absolutely convergent
because of the hyperbolicity of the fixed points joined by the
separatrices (see e.g.~\cite[Section 4.5]{GH}). Also notice that
although~\cite[Section 4.5]{GH} concerns transverse intersections of
homoclinic connections, the analysis applies verbatim to transverse intersections of heteroclinic connections.

More explicitly, the Melnikov functions are given by
\[
M_k(t_0)=\frac1{c_0^2}\int_{-\infty}^\infty R_k(t)^2\big[w_z(Z_k(t),R_k(t),t)\sin Z_k(t)-3R_k(t)J_0(R_k(t))w_r(Z_k(t),R_k(t),t)\big]\,dt\,,
\]
where $R_k(t)\equiv R_k(t;0,r_k)$ and $Z_k(t)\equiv Z_k(t;0,r_k)$. It is well known that the
existence of transverse intersections is independent of the choice of initial condition.

To analyze these Melnikov integrals, let us now choose the particular perturbation
\begin{equation}\label{eq_w}
w=J_1(r)\sin \te\, E_z+\frac{J_1(r)}{r}\cos \theta\, E_r-\frac{J_1'(r)\sin\theta}{r}\,E_\theta\,.
\end{equation}
It is easy to check that $\curl w=w$ in $\RR^3$; in fact
$w=(\curl\curl +\curl)(J_0(r),0,0)$ (or, to put it differently,
$w=U_{\vp' q(\xi_1)^{-1} p}$, where the distribution~$\vp'$ on the sphere~$\S$ is the Lebesgue measure
of the equator, normalized to unit mass). With this choice, the Melnikov functions take the form
\begin{align*}
c_0^2M_k(t_0)&=\int_{-\infty}^\infty R_k(t)^2\big[J_1(R_k(t))\sin Z_k(t)\sin (t+t_0)-3J_0(R_k(t))J_1(R_k(t))\cos (t+t_0)\big]\,dt \\
&=: a_k\sin t_0 +b_k \cos t_0\,,
\end{align*}
where the constants $a_k,b_k$ are given by the integrals
\begin{align*}
a_k&=\int_{-\infty}^\infty R_k(t)^2\big[J_1(R_k(t))\sin Z_k(t)\cos t+3J_0(R_k(t))J_1(R_k(t))\sin t\big]\,dt\,,\\
b_k&=\int_{-\infty}^\infty R_k(t)^2\big[J_1(R_k(t))\sin Z_k(t)\sin t-3J_0(R_k(t))J_1(R_k(t))\cos t\big]\,dt\,.
\end{align*}
Since the Hamiltonian function has the symmetry $H(-z,r)=H(z,r)$, it follows that $R_k(t)=R_k(-t)$ and $Z_k(t)=-Z_k(-t)$. This immediately yields that $a_1=a_2=0$. Moreover, it is not hard to compute the constants $b_1$ and $b_2$ numerically:
\[
\quad b_1=3.5508\dots\,,\quad b_2=0.2497\dots
\]

Therefore, the function $M_k(t_0)=b_k\cos t_0$ is a nonzero multiple
of the cosine, so it obviously has exactly two zeros in
the interval $[0,2\pi)$, which are nondegenerate. It then follows from
Lemma~\ref{L:melni} below that the two heteroclinic connections joining $p_{\pm}^\varepsilon$ intersect transversely. In turn, this implies~\cite[Theorem 26.1.3]{Wi} that each hyperbolic fixed point $p_{\pm}^\varepsilon$ has transverse homoclinic intersections, so
by the Birkhoff--Smale theorem~\cite[Theorem 5.3.5]{GH} the perturbed
system~\eqref{eq_pert}-\eqref{eq_pert2} (with $w$ given by Equation~\eqref{eq_w}) has a
compact hyperbolic invariant set on which the dynamics is
topologically conjugate to a Bernoulli shift. This set is contained in
a neighborhood of the heteroclinic cycle $\widetilde \Gamma_1 \cup
\widetilde \Gamma_2$, and hence in the planar domain $\cD$ where the
system is defined. This immediately implies that the vector field~$Y$
defined in Equation~\eqref{eq_Y}, which is the suspension of the
non-autonomous planar system~\eqref{eq_pert}, has a compact normally
hyperbolic invariant set~$K$ on which its time-$T$ flow is
topologically conjugate to a Bernoulli shift, where $T:= 2\pi N$ for
some positive integer $N>0$. The invariant set~$K$ is contained
in~$\Omega$ because it lies in a small neighborhood of the invariant
set $\Gamma_1\cup \Gamma_2$. Since the integral curves of~$X$ and~$Y$
are the same, up to a reparametrization, $K$~is also a chaotic invariant set of the Beltrami field~$X$ in~$\Omega$.

Finally, since $\RR^3\backslash \overline \Omega$ is connected, and of course the vector field $X$ satisfies the Beltrami equation in an open neighborhood of $\overline \Omega$, for each $\delta>0$, Proposition~\ref{P.approx} shows that
there is a Hermitian finite linear combination of spherical
harmonics~$\vp$ such that
\[
\|X-U_{\vp p}\|_{C^1(\Omega)}<\delta\,.
\]
If $\delta$ is small enough, the stability of transverse intersections
implies that the Beltrami field $U_{\vp p}$ has a compact chaotic
invariant set $K_\delta$ in a small neighborhood of~$K$ on which a suitable
reparametrization of its time-$T$ flow is conjugate to a Bernoulli shift, so the
proposition follows.
\end{proof}
\begin{corollary}\label{C.horseBelt}
There exists $R_0>0$ and $\delta>0$ such that $N^{\mathrm{h}}_w(R_0)\geq 1$ for any vector field $w$ such that $\|w-u_0\|_{C^k(B_{R_0})}<\delta$, provided that $k\geq 1$.
\end{corollary}
\begin{proof}
Taking $R_0$ so that the horseshoe of $u_0$ is contained in $B_{R_0}$, the result is a straightforward consequence of the lower semicontinuity of $N^{\mathrm{h}}_u(R)$, cf. Proposition~\ref{P.lsch}.
\end{proof}

To conclude, the following lemma gives the formula for the Melnikov
function that we employed in the proof of Proposition~\ref{P:chaos}
above. This is an expression for the Melnikov function of
perturbations of a planar system that is Hamiltonian with respect to
an arbitrary symplectic form. This is a minor generalization of the
well-known formulas~\cite[Theorem 4.5.3]{GH}  and~\cite[Equation~(23)]{Ho}, which assume that the symplectic form is the standard one.

\begin{lemma}\label{L:melni}
Let $Y_0$ be a smooth Hamiltonian vector field defined on a domain
$\cD\subset \RR^2$ with Hamiltonian function~$H$ and symplectic
form~$\omega$. Assume that this system has two hyperbolic fixed points
$p_\pm$ joined by a heteroclinic connection $\widetilde \Gamma$. Take
a smooth non-autonomous planar field~$Y_1$, which we assume
$2\pi$-periodic in time, and consider the perturbed system
$Y_0+\varepsilon Y_1+O(\varepsilon^2)$. Then the simple zeros of the Melnikov function
\[
M(t_0):=\int_{-\infty}^\infty \omega(Y_0,Y_1)|_{\gamma(t-t_0;p_0)}\, dt\,,
\]
where the integrand is evaluated at the integral curve
$\gamma(t-t_0;p_0)$ of~$Y_0$ parametrizing the separatrix $\widetilde
\Gamma$, give rise to a transverse heteroclinic intersection of the
perturbed system, for any small enough~$\ep$.
\end{lemma}
\begin{proof}
If $\varepsilon$ is small enough, the perturbed system has two
hyperbolic fixed points $p_\pm^\varepsilon$. To analyze how the
heteroclinic connection is perturbed, we take a point $p_0\in
\widetilde \Gamma$ and we compute the so-called displacement (or
distance) function $\Delta(t_0)$ on a section~$\Sigma$ based at~$p_0$
and transverse to~$\widetilde \Gamma$. Recall that the function $\varepsilon \Delta(t_0)$ gives the distance of the splitting, up to order $O(\varepsilon^2)$, between the corresponding stable and unstable manifolds of the perturbed system at the section $\Sigma$.

A standard analysis, cf.~\cite[Equation (22)]{Ho} or the proof of~\cite[Theorem 4.5.3]{GH}, yields the following formula for $\Delta(t_0)$:
\begin{equation}\label{eq:displ}
\Delta(t_0)=\frac{1}{|Y_0(p_0)|}\,\int_{-\infty}^\infty Y_1(\gamma(t-t_0))\times Y_0(\gamma(t-t_0)) e^{-\int_0^{t-t_0}\text{Tr}\, DY_0(\gamma(s))\,ds}\,dt\,,
\end{equation}
where we have omitted the dependence of the integral curve on the
initial condition $p_0\in\widetilde\Gamma$. Here we are using the
notation $X\times Y:= X_1Y_2-X_2Y_1$ for vectors $X,Y\in\RR^2$ and $\text{Tr}\, DY_0$ is the trace of the Jacobian matrix of the unperturbed field $Y_0$.

Take coordinates in~$\cD$, which we will call $(z,r)$ just as in the
proof of Proposition~\ref{P:chaos},
and write the symplectic form as $\omega=\rho(z,r)\,dz\wedge dr$,
where $\rho(z,r)$ is a smooth function that does not vanish. Let us
call here $\{e_z,e_r\}$ the basis
of vector fields dual to $\{dz,dr\}$ (which are usually denoted by
$\pd_z$ and $\pd_r$, as they correspond to the partial derivatives with
respect to the coordinates~$z$ and~$r$). The Hamiltonian field $Y_0$ reads in these coordinates as
$$
Y_0=\frac{1}{\rho(z,r)}\Big(\partial_r H\, e_z-\partial_zH\, e_r\Big)\,.
$$
Noting that
\[
Y_1(\gamma(t-t_0))\times Y_0(\gamma(t-t_0))=\frac{\omega(Y_0,Y_1)|_{\gamma(t-t_0)}}{\rho(\gamma(t-t_0))}
\]
and
\begin{align}
e^{-\int_0^{t-t_0}\text{Tr}\, DY_0(\gamma(s))\,ds}&=e^{\int_{0}^{t-t_0}Y_0(\gamma(s))\cdot \nabla \log \rho(\gamma(s))\,ds}\\
&=e^{\int_0^{t-t_0}\frac{d\log \rho(\gamma(s))}{ds}\,ds}=\frac{\rho(\gamma(t-t_0))}{\rho(p_0)}\,,
\end{align}
Equation~\eqref{eq:displ} implies that
\[
\Delta(t_0)=\frac{M(t_0)}{|Y_0(p_0)|\rho(p_0)}\,,
\]
so the claim follows because $M(t_0)$ coincides with the displacement
function up to a constant proportionality factor.
\end{proof}

\section{Asymptotics for random Beltrami fields on
  $\RR^3$}
\label{S.R3}

We are now ready to prove our main results about random Beltrami
fields on~$\RR^3$, Theorems~\ref{T.R3} and~\ref{T.R3zeros}.
To do this, as we saw in the two previous sections, we need to
handle sets that have a rather geometrically complicated structure, which gives rise to
several measurability issues. For this reason, we start this section
by proving a version of the Nazarov--Sodin sandwich
estimate~\cite[Lemma~1]{NS16} that circumvents some of these issues
and which is suitable for our purposes.

\subsection{A sandwich estimate for sets of points and for
  arbitrary closed sets}

For any subset~$\Ga\subset\RR^3$, we denote
by~$N(x,r;\Ga)$ the number of connected components of~$\Ga$ that are
contained in the ball~$B_r(x)$.  Also, if $\cX:=\{ x_j:j\in\cJ\}$,
where $x_j\in\RR^3$, is a countable set of points (which is not
necessarily a closed subset of~$\RR^3$), then we define
\[
  \cN(x,r;\cX):=\#[\cX\cap B_r(x)]
\]
as the number of points of~$\cX$ contained in the open
ball~$B_r(x)$. For the ease of notation, we will write
$N(r;\Ga):=N(0,r;\Ga)$ and similarly~$\cN(r;\cX)$. We remark that these numbers may be infinite.

\begin{lemma}\label{L.sandwich}
Let~$\Ga$ be any subset of~$\RR^3$ whose connected components are all closed and let $\cX:=\{
x_j:j\in\cJ\}$, with $x_j\in\RR^3$, be a countable set of points
of~$\RR^3$. Then the
functions $\cN(\cdot,r; \cX)$ and $N(\cdot,r; \Ga)$ are measurable,
and for any $0<r<R$ one has
\begin{align*}
	\int_{B_{R-r}} \frac{\cN(y, r; \cX)}{ \vol{B_r}}\, {d}y &\leq \cN(R; \cX)
	\leq \int_{B_{R+r}} \frac{\cN(y, r; \cX)}{\vol{B_r}}\, {d}y\,,\\
	\int_{B_{R-r}} \frac{N(y, r; \Gamma)}{\vol {B_r}}\, dy &\leq N(R; \Gamma) \,.
\end{align*}
\end{lemma}

\begin{proof}
  Let us start by noticing that
  \[
\cN(y,r;\cX)=\#\lbrace
j\in\cJ: x_j\in{B(y,r)} \rbrace= \sum_{j\in\cJ}\mathbbm{1}_{{B_r(x_j)}}(y)\,.
  \]
  As the ball ${B_r(x)}$ is an open set, it is clear that
$\mathbbm{1}_{{B_r(x)}}(\cdot)$  is a lower semicontinuous
function. Recall that lower semicontinuity is preserved under sums, and
that the
supremum of an arbitrary set (not necessarily countable) of lower semicontinuous functions is also
lower semicontinuous. Therefore, from the formula
$$
	\cN(\cdot,r;\cX)=\sup_{\cJ'}
        \sum_{j\in\cJ'}\mathbbm{1}_{{B_r(x_j)}}(\cdot)\,,
        $$
where $\cJ'$ ranges over all finite subsets of~$\cJ$, we deduce that
the function~$\cN(\cdot,r;\cX)$ is lower semicontinuous, and
therefore measurable. 

Now let $\cJ_R:=\{j\in\cJ: x_j\in B_R\}$ and note that
	$$
	\vol{B_r} \cN(R;
        \cX)=\sum_{j\in\cJ_R}\int_{B_{R+r}}\mathbbm{1}_{{B_r(x_j)}}(y)\,
        { d}y\,.
	$$
As we can interchange the sum and the integral by the monotone
convergence theorem and
$$ \sum_{j\in\cJ_R}\mathbbm{1}_{{B_r(x_j)}}(y)\leq \sum_{j\in\cJ}\mathbbm{1}_{{B_r(x_j)}}(y)=\cN(y,r;\cX)\,,$$	
one immediately obtains the upper bound for $\cN(R;\cX)$. Likewise, using now that
\begin{align*}
\vol{B_r} \cN(R;
\cX)&=
  \sum_{j\in\cJ_R}\int_{B_{R+r}}\mathbbm{1}_{{B_r(x_j)}}(y)\,dy\\
  &\geq\sum_{j\in\cJ_R}\int_{B_{R-r}}\mathbbm{1}_{{B_r(x_j)}}(y)\,
  { d}y\\
  &=\sum_{j\in\cJ}\int_{B_{R-r}}\mathbbm{1}_{{B_r(x_j)}}(y)\,
{ d}y
  =\int_{B_{R-r}}\cN(y,r;\cX)\, { d}y\,,
\end{align*}
we derive the lower bound. The sandwich estimate for~$\cN(R;\cX)$ is then proved.

Now let $\ga$ be a connected component of~$\Ga$, which is a closed set
by hypothesis. Since $\gamma\subset
B_r(y)$ if and only if $y\in B_r(x)$ for all $x\in\gamma$, one has
that
\begin{equation}\label{Nyr}
	N(y,r;\Gamma)=\sum_{\gamma\subset \Gamma}\mathbbm{1}_{\gamma^r}(y)\,,
      \end{equation}
      where the sum is over the connected components of~$\Ga$ and the
      set $\ga^r$ is defined, for each connected component~$\ga$
      of~$\Ga$, as
\[
\gamma^r:=\bigcap_{x\in\gamma}B_r(x)\,,
\]
that is, as the set of points in $\mathbb R^3$ whose distance to any point of $\gamma$ is less than $r$. Obviously, the set $\ga^r$ is open, so $\mathbbm{1}_{\ga^r}$ is lower
semicontinuous, and contained in the ball~$B_r(x_0)$, where $x_0$~is
any point of~$\ga$. Also notice that $\ga^r$ is not the empty set provided that $2r$ is larger than the diameter of $\ga$. Therefore, by the same argument as before, if follows from the
expression~\eqref{Nyr} that the function $N(\cdot,r;\Ga)$ is
measurable. If we now define the set $\Ga_R$ consisting of the
connected components of~$\Ga$ that are contained in the ball~$B_R$,
the same argument as before shows that
\begin{align*}
N(R;\Ga)&\geq \sum_{\ga\subset \Ga_R}\frac{1}{|\gamma^r|}\int_{B_{R+r}}\mathbbm{1}_{\gamma^r}(y)\,
          { d}y\\
  &\geq \sum_{\ga\subset \Ga_R}\frac{1}{|\gamma^r|}\int_{B_{R-r}}\mathbbm{1}_{\gamma^r}(y)\,
    { d}y\\
  &=\sum_{\ga\subset \Ga}\frac{1}{|\gamma^r|}\int_{B_{R-r}}\mathbbm{1}_{\gamma^r}(y)\,
        { d}y\\
        &\geq\int_{B_{R-r}} \frac{N(y,r;\Ga)}{\sup_{\ga\subset \Ga}|\gamma^r|}\, dy\\
  &\geq \int_{B_{R-r}} \frac{N(y,r;\Ga)}{|B_r|}\, dy\,.
\end{align*}
In the first inequality we are summing over components $\ga$ whose diameter is smaller than $2r$, and to pass to the last inequality we have used the obvious volume bound $|\gamma^r|\leq |B_r|$. Note that the proof of the upper bound for $\cN(R; \cX)$ does not apply in this
case, essentially because we do not have lower bounds for $|\ga^r|$ in terms of $|B_r|$.
\end{proof}

\subsection{Proof of Theorem~\ref{T.R3} and Corollary~\ref{C.R3}}

We are ready to prove Theorem~\ref{T.R3}. In fact, we will establish a
stronger result which permits to control the parameters of the
periodic orbits and the invariant tori. In what follows, we shall use the notation introduced in Sections~\ref{S.preliminaries} and~\ref{S.chaos} for the number of periodic orbits $\No(R;[\ga],\cI)$, the number of Diophantine toroidal sets $\Nt(R;[\cT],\cJ, V_0)$ (and the volume of the set of invariant tori $\Vt(R;[\cT],\cJ)$) and the number of horseshoes $\Nh(R)$. This is useful in itself,
since we showed in Section~\ref{per} that the quantity $\No(R;[\ga],\cI)$ is finite but this does not need to be the case if
one just counts~$\No(R;[\ga])$. Also, the choice of counting the
volume of invariant tori instead of its number (which one definitely
expect to be infinite) provides the trivial bound
$\Vt(R;[\cT],\cJ)\leq |B_R|$. Specifically, the result we prove is the following:

\begin{theorem}\label{T.R3b}
 Consider a closed curve~$\ga$ and an embedded torus~$\cT$ of~$\RR^3$. Then for any
  $\cI=(T_1,T_2,\La_1,\La_2)$, some
  $\cJ=(\om_1,\om_2,\tau_1,\tau_2)$ and some $V_0>0$, where
  \[
0<T_1<T_2\,,\quad 0<\La_1<\La_2\,,\quad 0<\om_1<\om_2\,,\quad 0<\tau_1<\tau_2\,,
  \]
  a Gaussian random Beltrami
  field~$u$ satisfies
  \begin{align*}
\liminf_{R\to\infty}
\frac{\Nh(R)}{|B_R|}&\geq\nuh\,,\\[1mm]
\liminf_{R\to\infty}
\frac{\Nt(R;[\cT],\cJ,V_0)}{|B_R|}&\geq\nut([\cT],\cJ,V_0)\,,\\[1mm]
\liminf_{R\to\infty} \frac{\No(R;[\ga],\cI)}{|B_R|}&\geq\nuo([\ga],\cI)
\end{align*}
with probability~$1$, with constants that are all positive. In particular, the topological
entropy of~$u$ is positive almost surely, and
\[
\liminf_{R\to\infty} \frac{\Vt(R;[\cT],\cJ)}{|B_R|}\geq V_0\,\nut([\cT],\cJ,V_0)\,,
\]
with probability~$1$.
\end{theorem}

\begin{proof}
For the ease of notation, let us denote by~$\Phi_R(u)$ the
quantities~$\Nh(R)$,  $\No(R;[\ga],\cI)$ and $\Nt(R;[\cT],\cJ,V_0)$, in
each case. Horseshoes are closed, and so are the set of periodic
orbits isotopic to~$\ga$ with parameters in~$\cI$ and the set of
closed invariant solid tori of the kind counted
by~$\Nt(R;[\cT],\cJ,V_0)$. Therefore, the lower bound for sets $\Gamma$ whose components are closed proved in
Lemma~\ref{L.sandwich}
ensures that, for any $0<r<R$,
\begin{align*}
\frac{\Phi_R(u)}{|B_R|}\geq \frac1{|B_R|}
  \int_{B_{R-r}}\frac{\Phi_r(\tau_yu)}{|B_r|}\, dy \geq \frac1{|B_R|}
  \int_{B_{R-r}}\frac{\Phi_r^m(\tau_yu)}{|B_r|}\, dy\,,
\end{align*}
where for any large~$m>1$ we have defined the truncation
\[
\Phi_r^m(w):= \min\{\Phi_r(w),m\}\,.
\]
We recall that the translation operator is defined as $\tau_yu(\cdot)=u(\cdot+y)$.

As the truncated random variable~$\Phi_r^m$ is in
$L^1(C^k(\RR^3,\RR^3),\mu_u)$ for any~$m$, one can consider the limit
$R\to\infty$ and apply
Proposition~\ref{P.ergodic} to conclude that
\begin{align*}
\liminf_{R\to\infty}\frac{\Phi_R(u)}{|B_R|}\geq
 \liminf_{R\to\infty} \frac{|B_{R-r}|}{|B_{R}|}\Bint_{B_{R-r}}\frac{\Phi_r^m(\tau_yu)}{|B_r|}\,
  dy = \frac1{|B_r|}\bE \Phi_r^m
\end{align*}
$\mu_u$-almost surely, for any~$r$ and $m$. Corollaries~\ref{C.perBelt}, \ref{C.torBelt}
and~\ref{C.horseBelt} imply that (for any~$\cI$ in the case of periodic
orbits, for some~$\cJ$ and some $V_0>0$ in the case of invariant tori,
and unconditionally in the case of horseshoes), there exists some
$r>0$, some $\de>0$ and a Beltrami field~$u_0$ such that
\[
\Phi_r(w)\geq 1
\]
for any divergence-free vector field~$w\in C^k(\RR^3,\RR^3)$ with
$\|w-u_0\|_{C^4(B_r)}<\de$. As
the random variable~$\Phi_r$ is nonnegative, and the measure~$\mu_u$ is
supported on Beltrami fields (cf. Proposition~\ref{P.support}), which are divergence-free, it is then immediate
that, when picking the parameters~$\cI$, $\cJ$ and~$V_0$ as above, one has for $k\geq 4$
\[
\bE \Phi_r^m\geq \mu_u\big(\{ w\in C^k(\RR^3,\RR^3):
\|w-u_0\|_{C^k(B_r)}<\de\}\big)=:\mathcal M(u_0,\de)\,.
\]
This is positive again by Proposition~\ref{P.support}. So defining the
constant, in each case, as
\begin{align*}
\nu:=\frac{\mathcal M(u_0,\de)}{|B_r|}>0
\end{align*}
the first part of the theorem follows.

Finally, the topological entropy of~$u$ is positive almost surely because~$u$ has a
horseshoe with probability~1, see Proposition~\ref{P.lsch}. The estimate for the growth of the volume of Diophantine invariant tori follows from the trivial lower bound
$$\Vt(R;[\cT],\cJ)>V_0\,\Nt(R;[\cT],\cJ,V_0)\,.$$
\end{proof}
\begin{remark}
A simple variation of the proof of Theorem~\ref{T.R3b} provides an
analogous result for links. We recall that a link $\mathcal L$ is a
finite set of pairwise disjoint closed curves in~$\RR^3$, which can be
knotted and linked among them. More precisely, if
$N^{\mathrm{l}}(R;[\mathcal L],\cI)$ is the number of unions of hyperbolic
periodic orbits of~$u$ that are  contained in $B_R$, isotopic to the
link $\mathcal L$, and whose periods and maximal Lyapunov exponents are in the intervals prescribed by $\cI$, then
\[
\liminf_{R\to\infty} \frac{N^{\mathrm{l}}(R;[\mathcal L],\cI)}{|B_R|}\geq\nu^{\mathrm{l}}([\mathcal L],\cI)>0\,.
\]
To apply the lower bound obtained in Lemma~\ref{L.sandwich} to
estimate the number of links, it is enough to transform each link into
a connected set by joining its different components by closed
arcs. The proof then goes exactly as in Theorem~\ref{T.R3b} upon
noticing that analogs of Proposition~\ref{P.perBelt} and Corollary~\ref{C.perBelt} also hold for
links (the proof easily carries over to this case).
\end{remark}

\begin{proof}[Proof of Corollary~\ref{C.R3}]
The corollary is now an immediate
consequence of the fact that the number of isotopy classes of closed
curves and embedded tori is countable. Indeed, by Theorem~\ref{T.R3},
with probability~1, a Gaussian random Beltrami field has infinitely
many horseshoes, an infinite volume of ergodic invariant tori isotopic
to a given embedded torus~$\cT$, and infinitely many periodic orbits
isotopic to a given closed curve~$\ga$. Since the countable
intersection of sets of probability~1 also has probability~$1$, the
claim follows.
\end{proof}

\subsection{Proof of Theorem~\ref{T.R3zeros}}\label{S.zeros}

We are now ready to prove the asymptotics for the number of
zeros of the Gaussian random Beltrami field~$u$. Let us start by
noticing that, almost surely, the zeros of~$u$ are nondegenerate. This
is because
\[
\mu_u\big(\big\{ w\in C^k(\RR^3,\RR^3): \det\nabla w(x)=0 \;\text{and}\;
w(x)=0 \text{ for some
  $x\in\RR^3$}\big\}\big)=0\,,
\]
which is a consequence of the boundedness of the probability density function (cf. Remark~\ref{R.rho}) and that $u$ is $C^\infty$ almost surely, see~\cite[Proposition 6.5]{AW09}. Hence the intersection of the
zero set
\[
  \cX_w:=\{x\in\RR^3: w(x)=0\}
\]
with a ball~$B_R$ is a finite set of points almost surely. The implicit function theorem then implies that these zeros are robust under $C^1$-small perturbations, so that with probability $1$, $\cN(R;\cX_v)\geq \cN(R;\cX_w)$ for any vector field $v$ that is close enough to $w$ in the $C^1$ norm. Summarizing, we have the following:

\begin{proposition}\label{P.lscz}
Almost surely, the functional $w\mapsto \cN(R;\cX_w)$ is
lower semicontinuous in the $C^k$~compact open topology for vector
fields, for any $k\geq 1$. Furthermore, $\cN(R;\cX_w)<\infty$ with probability $1$.
\end{proposition}

Since the variance $\bE[u(x)\otimes u(x)]$ is the identity matrix by
Corollary~\ref{C.diag}, the Kac--Rice formula~\cite[Proposition
6.2]{AW09} then enables us to compute the expected value
of the random variable
\begin{equation}\label{defPhir}
  \Phi_r(w):=\frac{\cN(r;\cX_w)}{|B_r|}
\end{equation}
as
\begin{align}\label{ENz}
  \bE\Phi_r&=\Bint_{B_r}\bE\{|\det\nabla w(x)|: w(x)=0\}\, \rho(0)\, dx\notag\\
  &=(2\pi)^{-\frac32} \bE\{|\det\nabla w(x)|: w(x)=0\}\,.
\end{align}
Here we have used that the above conditional expectation is
independent of the point~$x\in\RR^3$ by the translational invariance of the probability measure. We recall that the probability density function~$\rho(y):=
(2\pi)^{-\frac32}\, e^{-\frac12|y|^2}$ was
introduced in Remark~\ref{R.rho}.

To compute the above conditional expectation value, one can argue as follows:

\begin{lemma}\label{L.bE}
  For any~$x\in\RR^3$,
  \[
\bE\{ |\det \nabla u(x)| : u(x)=0\}= (2\pi)^{\frac32}\nuz\,,
\]
where the constant~$\nuz$ is given by~\eqref{nuz}.
  \end{lemma}

  \begin{proof}
Let us first reduce the computation of the conditional expectation to
that of an ordinary expectation by introducing a new random
variable~$\zeta$. Just like~$\nabla u(x)$, this new variable takes
values in the space of $3\times 3$~matrices, which we will identify
with~$\RR^9$ by labeling the matrix entries as
\begin{equation}\label{zeta}
\zeta=: \left(
\begin{array}{ccc}
\zeta _1 & \zeta _2 & \zeta _3 \\
\zeta _4 & \zeta _5 & \zeta _6 \\
\zeta _7 & \zeta _8 & \zeta _9 \\
\end{array}
\right)\,.
\end{equation}
This variable is defined as
\begin{equation}\label{defzeta}
\zeta:= \nabla u(x)-B u(x)\,,
\end{equation}
where the linear operator~$B$ (which is a $9\times 3$~matrix if we
identify $\nabla u(x)$ with a vector in~$\RR^9$) is chosen so that the
covariance matrix of~$u(x)$ and~$\zeta$ is~0:
\[
B:=\bE (\nabla u(x)\otimes u(x))\big[\bE ( u(x)\otimes u(x))\big]^{-1}=\bE (\nabla u(x)\otimes u(x))\,.
\]
Here we have used that the second matrix is in fact the identity by Corollary~\ref{C.diag}.
An easy computation shows that then
\[
\bE(\zeta\otimes u(x))=0\,;
\]
as $u(x)$ and $\zeta$ are Gaussian vectors with zero mean, this condition
ensures that they are independent random variables. Therefore, we can use the
identity~\eqref{defzeta} to write the conditional expectation as
\[
\bE\{ |\det \nabla u(x)| : u(x)=0\}= \bE\{ |\det [\zeta+ Bu(x)]| : u(x)=0\}=\bE|\det\zeta|\,.
    \]

Our next goal is to compute the covariance matrix of~$\zeta$ in closed
form, which will enable us to find the expectation of
$|\det\zeta|$. By definition,
\begin{align*}
\bE(\zeta\otimes \zeta)&= \bE[(\nabla u(x)-B u(x))\otimes (\nabla
                          u(x)-B u(x))]\\
  &=\bE[\nabla u(x)\otimes \nabla u(x)]- \bE[\nabla u(x)\otimes u(x)]\,
  \bE[ u(x) \otimes \nabla u(x)]\,.
\end{align*}
The basic observation now is that, for any Hermitian polynomials in three variables
$q(\xi)$ and $q'(\xi)$, the argument that we used to establish the
formula~\eqref{formulakappa} and Corollary~\ref{C.diag} shows that
\begin{align*}
\bE[(q(D) u_j(x))\, (q'(D)u_k(x))]&= \bE[q(D_x) u_j(x)\,
                                      \overline{q'(D_y)u_k(y)}]|_{y=x}\\
  &= \int_\S q(\xi)\, q'(-\xi)\, p_j(\xi)\, \overline{p_k(\xi)}\,
    e^{i\xi\cdot (x-y)}\, d\si(\xi)\bigg|_{y=x}\\
  &= \int_\S q(\xi)\, q'(-\xi)\, p_j(\xi)\, \overline{p_k(\xi)}\,
    d\si(\xi)\,.
\end{align*}
Here we have used that $q'(D)u_k$ is real-valued because $q'$ is Hermitian. As all the matrix integrals in the calculation of $\bE(\zeta\otimes \zeta)$ are of this form with $q(\xi)=i\xi$ or $1$, the computation again boils down to evaluating integrals of the form
$\int_\S \xi^\al\, d\si(\xi)$, which can be computed using the formula~\eqref{Folland}.

Tedious but straightforward computations then yield the following
explicit formula for the covariance matrix of~$\zeta$:
\[
  \Si:= \bE(\zeta\otimes\zeta)=
\left(
\begin{array}{ccccccccc}
\frac{5}{21} & 0 & 0 & 0 & -\frac{5}{42} & 0 & 0 & 0 & -\frac{5}{42} \\
0 & \frac{11}{84} & 0 & \frac{11}{84} & 0 & 0 & 0 & 0 & 0 \\
0 & 0 & \frac{11}{84} & 0 & 0 & 0 & \frac{11}{84} & 0 & 0 \\
0 & \frac{11}{84} & 0 & \frac{11}{84} & 0 & 0 & 0 & 0 & 0 \\
-\frac{5}{42} & 0 & 0 & 0 & \frac{3}{14} & 0 & 0 & 0 & -\frac{2}{21} \\
0 & 0 & 0 & 0 & 0 & \frac{13}{84} & 0 & \frac{13}{84} & 0 \\
0 & 0 & \frac{11}{84} & 0 & 0 & 0 & \frac{11}{84} & 0 & 0 \\
0 & 0 & 0 & 0 & 0 & \frac{13}{84} & 0 & \frac{13}{84} & 0 \\
-\frac{5}{42} & 0 & 0 & 0 & -\frac{2}{21} & 0 & 0 & 0 & \frac{3}{14} \\
\end{array}
\right)
\]
Note that this matrix is not invertible: it has rank~5, and an orthogonal basis for the ($4$-dimensional) kernel is
$$
\{ e_1+e_5+e_9,\; e_2-e_4,\; e_3-e_7, \; e_6-e_8\}\,,
$$
where $\{e_j\}_{j=1}^9$ denotes the canonical basis of~$\RR^9$. As we
are dealing with Gaussian vectors, this is equivalent to the assertion
that
\begin{equation}\label{zetadep}
\zeta_1+\zeta_5+\zeta_9=0\,,\quad \zeta_2=\zeta_4\,,\quad
\zeta_3=\zeta_7\,,\quad \zeta_6=\zeta_8\
\end{equation}
almost surely (which amounts to saying that~$\zeta$ is a traceless
symmetric matrix). Notice that these equations define a $5$-dimensional subspace orthogonal to the kernel of $\Si$. The remaining random variables $\zeta':=
(\zeta_1,\zeta_2,\zeta_3,\zeta_5,\zeta_6)$ are independent Gaussians
with zero mean and covariance matrix
\[
\Sigma':=\bE(\zeta'\otimes\zeta')=\left(
\begin{array}{ccccc}
\frac{5}{21} & 0 & 0 & -\frac{5}{42} & 0 \\
0 & \frac{11}{84} & 0 & 0 & 0 \\
0 & 0 & \frac{11}{84} & 0 & 0 \\
-\frac{5}{42} & 0 & 0 & \frac{3}{14} & 0 \\
0 & 0 & 0 & 0 & \frac{13}{84} \\
\end{array}
\right)
\]
By construction, $\Si'$ is an invertible matrix, so we can immediately
write down a formula for the expectation value of~$|\det\zeta|$:
\begin{align*}
\bE|\det\zeta|&=(2\pi)^{-\frac52} (\det\Si')^{-\frac12}\int_{\RR^5} \left| \det\left(
\begin{array}{ccc}
\zeta _1 & \zeta _2 & \zeta _3 \\
\zeta _2 & \zeta _5 & \zeta _6 \\
\zeta _3 & \zeta _6 & -\zeta_1-\zeta_4 \\
\end{array}
  \right)\right|\, e^{-\frac12\zeta'\cdot \Si'^{-1}\zeta'}\, d\zeta'\\
  &=(2\pi)^{-\frac52} (\det\Si')^{-\frac12}\int_{\RR^5} | Q(\zeta')|\, e^{-\frac12\zeta'\cdot \Si'^{-1}\zeta'}\, d\zeta'\,,
\end{align*}
with the cubic polynomial~$Q$ being defined as in~\eqref{Q(z)}.
Since $\frac12 \zeta'\cdot\Si'^{-1}\zeta'= \widetilde Q(\zeta')$, where the
quadratic polynomial~$\widetilde Q$ was defined in~\eqref{tildeQ}, and
\[
  \det\Si'=\frac{5\cdot 143^2}{{2^8}\cdot{21^5}}\,,
\]
we therefore have
\[
\bE|\det\zeta|=(2\pi)^{\frac32}\nuz\,.
\]
The result then follows.
\end{proof}

\begin{remark}
If one keeps track of the
connection between~$\zeta$ and~$\nabla u(x)$, it is not hard to see
that the first condition $\zeta_1+\zeta_5+\zeta_9=0$ in~\eqref{zetadep} is equivalent to
$\Div u(x)=0$, while the remaining three just mean that $\curl u(x)= u(x)$, at the points $x\in\RR^3$ where $u(x)=0$.
\end{remark}

In particular, this shows that $\Phi_R\in
L^1(C^k(\RR^3,\RR^3),\mu_u)$. For the ease of notation, let us define the ergodic mean operator
\[
\cA_R\Phi(w):= \frac1{|B_R|} \int_{B_R} \Phi(\tau_y w)\,{d}y\,.
\]
Since~$\cN(R,\cX_w)$ is finite almost surely, cf. Proposition~\ref{P.lscz}, the sandwich estimate proved in
Lemma~\ref{L.sandwich} implies that, almost surely,
\[
\frac1{|B_R|}\int_{B_{R-r}}\Phi_r(\tau_yw)\, dy\leq \Phi_R(w)\leq \frac1{|B_R|}\int_{B_{R+r}}\Phi_r(\tau_yw)\, dy
\]
for any~$0<r<R$. Therefore, and using that $|B_{R\pm r}|/|B_R|=(1\pm
r/R)^3$, one has
\[
|\Phi_R-\cA_R\Phi_r|\leq \bigg|\bigg(1+\frac rR\bigg)^3
\cA_{R+r}\Phi_r-\cA_R\Phi_r\bigg| + \bigg|\bigg(1-\frac rR\bigg)^3
\cA_{R-r}\Phi_r-\cA_R\Phi_r\bigg| \,.
\]
For fixed~$r$, Equation~\eqref{ENz} and Proposition~\ref{P.ergodic} ensure that
\begin{equation}\label{tendsnuz}
  \cA_R\Phi_r\Lone\bE\Phi_r=\nuz
\end{equation}
as $R\to\infty$; also, note that the
limit (which is independent of~$r$) has been computed in
Lemma~\ref{L.bE} above.

Therefore, if we let~$R\to\infty$ while~$r$ is held fixed, the RHS of
the estimate before Equation~\eqref{tendsnuz} tends to~0 $\mu_u$-almost surely and
in~$L^1(\mu_u)$, so that
\[
\Phi_R-\cA_R\Phi_r\Lone0
\]
as $R\to\infty$. As $\cA_R\Phi_r\Lone\nuz$ by~\eqref{tendsnuz},
Theorem~\ref{T.R3zeros} is proven.

\section{The Gaussian ensemble of Beltrami fields on the torus}
\label{S.BFtorus}

\subsection{Gaussian random Beltrami fields on the torus}

As introduced in Section~\ref{S.introT}, a Beltrami field on the flat 3-torus $\TT^3:=(\RR/2\pi\mathbb Z)^3$
(or, equivalently, on the cube of~$\RR^3$ of side length~$2\pi$ with
periodic boundary conditions) is a vector field on~$\TT^3$ satisfying
the equation
\[
\curl v= \la v
\]
for some real number~$\la\neq0$. To put it differently, Beltrami fields
on the torus are the eigenfields of the curl operator. It is easy to
see that such an eigenfield is divergence-free and has zero mean, that
is, $\int_{\TT^3} v\, dx=0$.

Since $\De v+\la^2 v=0$, it is well-known (see e.g.~\cite{Adv}) that the spectrum of
the curl operator on the $3$-torus consists of the numbers of the form
$\la=\pm|k|$ for some vector with integer coefficients $k\in\mathbb
Z^3$. For concreteness, we will henceforth assume that $\la>0$; the case of
negative frequencies is completely analogous. Since $k$ has integer
coefficients, one can label the positive eigenvalues of curl by a
positive integer~$L$ such that $\la_L=L^{1/2}$. Let us define
\[
\cZ_L:= \{k\in\mathbb Z^3: |k|^2=L\}
\]
and note that the set~$\cZ_L$ is invariant under reflections (i.e.,
$-k\in\cZ_L$ if $k\in\cZ_L$).

The Beltrami fields corresponding to the eigenvalue~$\la_L$ must be of
the form
\begin{equation}\label{V1}
v=\sum_{k\in\cZ_L} V_k\, e^{ik\cdot x}\,,
\end{equation}
for some $V_k\in\CC^3$. Conversely, this expression defines a
Beltrami field with frequency~$\la_L$ if and only if $V_k=
\overline{V_{-k}}$ (which ensures that~$v$ is real valued) and
\[
\frac{ik}{L^{1/2}}\times V_k = V_k\,.
\]
Since $|k|=L^{1/2}$, we infer from the proof of Proposition~\ref{P.p(xi)}
that the vector~$V_k$ must be of the form
\begin{equation}\label{V2}
V_k= \al_k\, p(k/L^{1/2})
\end{equation}
unless $k=(\pm L^{1/2},0,0)$. Here $\al_k\in\CC$ is an arbitrary complex
number and the Hermitian vector field~$p(\xi)$ was defined in~\eqref{p(xi)}.

The multiplicity of the eigenvalue~$\la_L$ is given by the cardinality
$d_L:=\#\cZ_L$. By Legendre's three-square theorem, $\cZ_L$ is
nonempty (and therefore~$\la_L$ is an eigenvalue of the curl operator)
if and only if~$L$ is not of the form $4^a(8b+7)$ for nonnegative
integers~$a$ and~$b$.

Based on the formulas~\eqref{V1}-\eqref{V2}, we are now ready to
define a Gaussian random Beltrami field on the torus with
frequency~$\la_L$ as
\begin{equation}\label{uL}
u^L(x):=\bigg(\frac{2\pi}{d_L}\bigg)^{1/2}\sum_{k\in\cZ_L} a^L_k\, p(k/L^{1/2})\, e^{ik\cdot x}\,,
\end{equation}
where the real and imaginary parts of the complex-valued random
variable~$a^L_k$ are standard Gaussian variables. We also assume that
these random variables are independent except
for the constraint $a^L_k=\overline{a^L_{-k}}$. The inessential
normalization factor $(2\pi/d_L)^{1/2}$ has been introduced for later
convenience.

Note that $u^L(x)$ is a smooth $\RR^3$-valued function of
the variable~$x$, so it induces a Gaussian probability measure~$\mu^L$
on the space of $C^k$-smooth vector fields on the torus,
$C^k(\TT^3,\RR^3)$. As before, we will always assume that $k\geq4$ to
apply results from KAM~theory. We will also employ the rescaled
Gaussian random field
\[
u^{L,z}(x):= u^L\bigg(z+\frac {x}{L^{1/2}}\bigg)
\]
for any fixed point $z\in \TT^3$.

\subsection{Estimates for the rescaled covariance matrix}

In what follows, we will restrict our attention to the positive
integers~$L$, which we will henceforth call {\em admissible}\/, that
are not congruent with 0, 4 or 7 modulo~8. When~$L$ is congruent
with~0 or~7 modulo~8, Legendre's three-square theorem immediately
implies that~$\cZ_L$ is empty. The reason to rule out numbers
congruent with~4 modulo~8 is more subtle: a deep theorem of
Duke~\cite{Duke}, which addresses a question raised by Linnik, ensures that the set~$\cZ_L/L^{1/2}$ becomes uniformly
distributed on the unit sphere as $L\to\infty$ through
integers that are congruent to 1, 2, 3, 5 or~6 modulo~8. This ensures
that
\begin{equation}\label{unif}
\frac{4\pi}{d_L}\sum_{k\in\cZ_L} \phi(k/L^{1/2})\to \int_\S \phi(\xi)\, d\si(\xi)
\end{equation}
as $L\to\infty$ through admissible values, for any continuous function~$\phi$ on $\S$. A particular case is when
$L$ goes to infinity through squares of odd values, that is, when $L=(2m+1)^2$ and
$m\to\infty$.

The covariance kernel of the Gaussian random variable~$u^L$ is the matrix-valued function
\[
\ka^L(x,y):= \bE^L[u^L(x)\otimes u^L(y)]\,.
\]
Following Nazarov and Sodin~\cite{NS16}, we will be most interested in
the covariance kernel of the rescaled field~$u^{L,z}$ at a
point~$z\in\TT^3$, which is given by
\[
\ka^{L,z}(x,y)= \bE^L\bigg[u^L\bigg(z+\frac {x}{L^{1/2}}\bigg)\otimes u^L\bigg(z+\frac {y}{L^{1/2}}\bigg)\bigg]\,.
\]
The following proposition ensures that, for large admissible frequencies~$L$, the rescaled
covariance kernel, and suitable generalizations thereof, tend to those of a Gaussian random Beltrami
field on~$\RR^3$, $\ka(x,y)$, defined in~\eqref{kappa}:

\begin{proposition}\label{P.covL}
  For any $z\in\TT^3$, the rescaled covariance kernel
  $\ka^{L,z}(x,y)$ has the following properties:
  \begin{enumerate}
  \item It is invariant under translations and independent
    of~$z$. That is, there exists some function $\varkappa^L$ such
    that
    \[
\ka^{L,z}(x,y)=\varkappa^L(x-y)\,.
    \]

    \item Given any compact set
      $K\subset\RR^3$, the covariance kernel satisfies
  \[
\ka^{L,z}(x,y)\to \ka(x,y)
  \]
  in $C^s(K\times K)$ as $L\to\infty$ through admissible values.
  \end{enumerate}
  \end{proposition}

\begin{proof}
  Let~$\al$, $\be$ be any multiindices, and recall the operator $D=-i\nabla$ introduced in Section~\ref{S.random}. By definition, and using the fact that $u^L$ is real,
  \begin{align*}
&D_x^\al D_y^\be \ka^{L,z}(x,y)= \bE^L\bigg[D_x^\al u^L\bigg(z+\frac {x}{L^{1/2}}\bigg)\otimes
              D_y^\be u^L\bigg(z+\frac {y}{L^{1/2}}\bigg)\bigg]\\
            &= \bE^L\bigg[D_x^\al u^L\bigg(z+\frac {x}{L^{1/2}}\bigg)\otimes
             D_y^\be\overline{ u^L\bigg(z+\frac {y}{L^{1/2}}\bigg)}\bigg]\\
    &=\frac{2\pi}{d_L}\sum_{k\in\cZ_L}\sum_{k'\in\cZ_L} \bE^L
      (a^L_k\overline{a^L_{k'}})\, p\bigg(\frac {k}{L^{1/2}}\bigg)\otimes\overline{p\bigg(\frac{k'}{L^{1/2}}\bigg)}\,
      \bigg(\frac {k}{L^{1/2}}\bigg)^\al
      \bigg(\frac{-k'}{L^{1/2}}\bigg)^\be\,  e^{ik\cdot (z+\frac {x}{L^{1/2}})-ik'\cdot
      (z+\frac {y}{L^{1/2}})}\,.
  \end{align*}
  The independence properties of the Gaussian variables~$a^L_k$ (which have zero mean)
  imply that $\bE^L
      (a^L_k\overline{a^L_{k'}})=0$ if $k'\not\in \{k,-k\}$. When
      $k'=k$ one has
      \[
        \bE^L [|a^L_k|^2]= \bE^L[ (\Real a^L_k)^2]+ \bE^L[ (\Imag a^L_k)^2]=2\,,
      \]
      and when $k'=-k$,
\[
        \bE^L [(a^L_k)^2]= \bE^L[ (\Real a^L_k)^2]- \bE^L[ (\Imag
        a^L_k)^2]+2i \,\bE^L[ (\Real a^L_k)(\Imag a^L_k)]=0\,.
      \]
      Therefore, $\bE^L
      (a^L_k\overline{a^L_{k'}})=2\de_{k k'}$ and we obtain
      \begin{align*}
       D_x^\al D_y^\be \ka^{L,z}(x,y)=\frac{4\pi}{d_L}\sum_{k\in\cZ_L}p\bigg(\frac {k}{L^{1/2}}\bigg)\otimes\overline{p\bigg(\frac {k}{L^{1/2}}\bigg)}\,
      \bigg(\frac {k}{L^{1/2}}\bigg)^\al \bigg(-\frac {k}{L^{1/2}}\bigg)^\be\, e^{ik\cdot(x-y)/L^{1/2}}\,.
      \end{align*}
In particular, this formula shows that
$\ka^{L,z}(x,y)$ is independent of~$z$ and translation-invariant.

      Using now the fact that $\cZ_L$ becomes uniformly distributed on~$\S$
      as~$L\to\infty$ through admissible values, we obtain via
      Equation~\eqref{unif} that
      \begin{align*}
D_x^\al D_y^\be \ka^{L,z}(x,y)&\to\int_\S \xi^\al (-\xi)^\be \,
                              p(\xi)\otimes\overline{p(\xi)}\, e^{i\xi\cdot (x-y)}\, d\si(\xi)\\
  &= D_x^\al D_y^\be  \int_\S
                              p(\xi)\otimes\overline{p(\xi)}\, e^{i\xi\cdot (x-y)}\, d\si(\xi)      \,.
      \end{align*}
      By Proposition~\ref{P.kernel}, the RHS equals
      $D_x^\al D_y^\be\ka(x,y)$, so the result follows.
    \end{proof}

   \subsection{A convergence result for probability measures}

We shall next present a result showing that the
probability measure defined by the rescaled field~$u^{L,z}$ converges,
as $L\to\infty$, to that defined by the Gaussian random Beltrami field
on~$\RR^3$, $u$, on compact sets of~$\RR^3$:

\begin{lemma}\label{L.convmu}
Fix some $R>0$ and denote by~$\mu^{L,z}_R$
and~$\mu_{u,R}$, respectively, the probability measures on $C^k(B_R,\RR^3)$ defined by
the Gaussian random fields~$u^{L,z}$ and~$u$. Then the
measures~$\mu^{L,z}_R$ converge weakly to~$\mu_{u,R}$ as $L\to\infty$
through the admissible integers.
\end{lemma}

\begin{proof}
Let us start by noting that all the finite dimensional distributions of
the fields $u^{L,z}$ converge to those of~$u$ as $L\to\infty$. Specifically, consider any finite number of points $x^1,\dots ,
  x^n\in\RR^3$, any indices $j^1,\dots, j^n\in\{1,2,3\}$, and any
  multiindices with $|\al^j|\leq k$. Then it is not hard to see that the Gaussian vectors of zero
expectation
\[
(\pd^{\al^1}u_{j^1}^{L,z}(x^1), \dots, \pd^{\al^n}u_{j^n}^{L,z}(x^n))\in\RR^{n}
\]
converge in distribution to the Gaussian vector
\begin{equation}\label{Gvect2}
(\pd^{\al^1}u_{j^1}(x^1),\dots, \pd^{\al^n}u_{j^n}(x^n))
\end{equation}
as $L\to\infty$. This follows from the fact that their probability
density functions are completely determined by the $n\times n$
variance matrix
\[
 \Si^L:=
 \Big(\pd_x^{\al^l}\pd_y^{\al^m}\ka_{j^lj^m}^{L,z}(x,y)\big|_{(x,y)=(x^l,x^m)}\Big)_{1\leq
   l,m\leq n}\,,
\]
which converges to
$\Si:=(\pd_x^{\al^l}\pd_y^{\al^m}\ka_{j^lj^m}(x,y)|_{(x,y)=(x^l,x^m)})$ as
$L\to\infty$ by Proposition~\ref{P.covL}. The latter, of course, is
the covariance matrix of the Gaussian vector~\eqref{Gvect2}.

It is well known that this convergence of arbitrary Gaussian vectors
is not enough to conclude that $\mu^{L,z}_R$ converges weakly
to~$\mu_{u,R}$. However, notice that, for
any integer~$s\geq0$, the mean of the $H^s$-norm of~$u^{L,z}$ is uniformly bounded:
\begin{align*}
\bE^{L,z}\|w\|_{H^s(B_R)}^2&=\sum_{|\al|\leq s}\bE\int_{B_R}
                             |D^\al u^{L,z}(x)|^2\, dx\\
  &=\sum_{|\al|\leq s}\int_{B_R}  \Tr\Big(D_x^\al D_y^\al
    \ka^{L,z}(x,y)\big|_{y=x}\Big)\, dx\\
  &\xrightarrow[L\to\infty]{\phantom{L}}
 \sum_{|\al|\leq s}\int_{B_R} \Tr \Big(D_x^\al
    D_y^\al\ka(x,y)\big|_{y=x}\Big)\, dx<M_{s,R}\,.
\end{align*}
To pass to the last line, we have used Proposition~\ref{P.covL} once
more. As the constant~$M_{s,R}$ is independent of~$L$, Sobolev's
inequality ensures that
\[
\sup_L\bE^{L,z} \|w\|_{C^{k+1}(B_R)}^2\leq C\sup_L\bE^{L,z} \|w\|^2_{H^{k+3}(B_R)}<M
\]
for some constant~$M$ that only depends on~$R$. For any~$\ep>0$, this implies that for all admissible $L$ large enough
\[
\mu^{L,z}_R\big(\big\{ w\in C^k(B_R,\RR^3): \|w\|_{C^{k+1}(B_R)}^2> M/\ep\big\}\big)<\ep\,.
\]
As the set $\{ w\in C^k(B_R,\RR^3): \|w\|_{C^{k+1}(B_R)}^2\leq
M/\ep\}$ is compact by the Arzel\`a--Ascoli theorem, we conclude
that the sequence
of probability measures $\mu^{L,z}_R$ is tight.
Therefore, a straightforward extension to jet spaces of the classical
results about the
convergence of probability measures on the space of continuous
functions~\cite[Theorem 7.1]{Bill}, carried out in~\cite{Wilson},
permits to conclude that $\mu^{L,z}_R$ indeed converges weakly to~$\mu_{u,R}$ as $L\to\infty$. The lemma is then proven.
\end{proof}

    \subsection{Proof of Theorem~\ref{T.torus}}

   We are now ready to prove our asymptotic estimates for high-frequency
   Beltrami fields on the torus. The basic idea is that, by the
   definition of the rescaling,
   \[
\mu^L\big(\big\{ w\in C^k(\TT^3,\RR^3):
N^{\mathrm{h}}_w >m\big\}\big)\geq \mu^{L,z}_R\big(\big\{ w\in C^k(B_R,\RR^3):
N^{\mathrm{h}}_w(r)>m\big\}\big)
   \]
  provided that $r<R<L^{1/2}$: this just means that the number of horseshoes
  that~$u^L$ has in the whole torus is certainly not
  less than those that are contained in a ball centered at any given
  point $z\in\TT^3$ of radius $r/L^{1/2}<1$. The same is clearly true as
  well when one counts invariant solid tori, periodic orbits or zeros instead.

For the ease of notation, let us denote by~$\Phi_r(w)$ the
quantity~$N^{\mathrm{h}}_w(r)$,  $N^{\mathrm{t}}_w(r;[\cT],\cJ,V_0)$, $N^{\mathrm{o}}_w(r;[\ga],\cI)$ or $N^{\mathrm{z}}_w(r)$ (that is, the number of
nondegenerate zeros of~$w$ in~$B_r$), in
each case. See Sections~\ref{S.preliminaries} and~\ref{S.chaos} for precise definitions. We recall that $N^{\mathrm{z}}_w(r)=\cN(r;\cX_w)$ with probability~$1$, cf. Section~\ref{S.zeros}. Theorems~\ref{T.R3b} (for periodic orbits, invariant tori and horseshoes) and~\ref{T.R3zeros} (for zeros) ensure that, given any $m_1>0$, any~$\de_1>0$, any
closed curve~$\ga$ and any embedded torus~$\cT$, one can find some
parameters~$\cI$, $\cJ$, $V_0$ and $r>0$ such that
\[
\mu_u\big(\big \{w\in C^k(\RR^3,\RR^3): \Phi_r(w)>m_1 \big\}\big)>1-\de_1\,.
\]
Of course, here we are simply using that the volume
$|B_r|$, which appears in the statements of Theorems~\ref{T.R3b} and~\ref{T.R3zeros} but not
here, can be made arbitrarily large by taking a large~$r$.

Let us now fix some~$R>r$ and some point $z\in\TT^3$. We showed in Propositions~\ref{P.lscpo},~\ref{P.lscit},~\ref{P.lsch} and~\ref{P.lscz} that the functionals that we are now denoting
by~$\Phi_r$ are lower semicontinuous on the space~$C^k(\RR^3,\RR^3)$ of divergence-free fields for
$k\geq4$. This implies that the set
\[
\Omega_{r,R,m_1}:=\{w\in C^k(B_R,\RR^3): \Phi_r(w)>m_1\}
\]
is open in~$C^k(B_R,\RR^3)$. Lemma~\ref{L.convmu} ensures that the
measure $\mu^{L,z}_R$ converges weakly to~$\mu_{u,R}$ as $L\to\infty$
through the admissible integers. As the set~$\Omega_{r,R,m_1}$ is open,
this is well known to imply (see e.g.~\cite[Theorem 2.1.iv]{Bill}) that
\begin{align*}
\liminf_{L\to\infty}\mu^{L,z}_R(\Omega_{r,R,m_1})&\geq
                                                 \mu_{u,R}(\Omega_{r,R,m_1})\\
                                               &=\mu_u\big( \big\{w\in
                                                 C^k(\RR^3,\RR^3):
                                                 \Phi_r(w)>m_1 \big\}\big)\\
  &>1-\de_1\,.
\end{align*}
We observe that $\de_1>0$ can be taken
arbitrarily small if $r$ is large enough (and $r/L^{1/2}<R/L^{1/2}<1$). Now, for any $A\geq 1$ and $L$ large enough, we can take $A$ pairwise disjoint balls in $\TT^3$ of radius $r/L^{1/2}<A^{-1/3}$ centered at points $\{z^a\}_{a=1}^A\subset \TT^3$. Setting $m:=Am_1$, the previous analysis, which is independent of the point $z$, readily implies that
\[
\mu^L\big(\big\{ w\in C^k(\TT^3,\RR^3):
N^{\mathrm{X},\mathrm{e}}_w >m\big\}\big)\geq 1-2A\de_1>1-\de\,,
   \]
where the superscript~X stands for $\mathrm{h,t,o}$ or $\mathrm{z}$, thus proving the part of the statement concerning the number of approximately equidistributed horseshoes, invariant tori isotopic to~$\cT$, periodic
orbits isotopic to~$\ga$ or zeros. In fact, concerning invariant tori, we observe that obviously $N^{\mathrm{t}}_w(r;[\cT])=\infty$ if
$N^{\mathrm{t}}_w(r;[\cT],\cJ,V_0)\geq1$. Since the previous estimate ensures that $N^{\mathrm{t}}_w(r;[\cT],\cJ,V_0)>m_1$ with probability $1$ as $L\to\infty$, we infer that the probability of having an infinite number of (Diophantine) invariant tori isotopic
to~$\cT$ also tends to~1 as $L\to\infty$ through the admissible integers. However this does not provide any information about the expected volume of
the invariant tori.

The result about the topological entropy follows from the following
observation. If we denote by $\phi^L_t$ the time-$t$ flow of the Beltrami field $u^L(z+\cdot)$, and by $\phi_t$ the flow of the rescaled field $u^{L,z}$, it is evident that
\[
\phi^L_t=\frac{1}{L^{1/2}}\phi_{L^{1/2}t}\,.
\]
Then, the topological entropy $\htop(u^L)$, which is defined as the entropy of its time-$1$ flow, satisfies
\begin{align}\label{eq.entropy}
\htop(u^L)&=\htop(\phi^L_1)=\htop\Big(\frac{1}{L^{1/2}}\phi_{L^{1/2}}\Big)=\htop(\phi_{L^{1/2}})=L^{1/2}\htop(\phi_1)\\
&=L^{1/2}\htop(u^{L,z})\,.
\end{align}
In the third equality we have used that the topological entropy does
not depend on the space scale (or equivalently, on the metric), and in
the fourth equality we have used Abramov's well-known formula (see
e.g.~\cite{GM10}). Since the rescaled field has a horseshoe in a ball
of radius~$r$ with probability $1$ as $L\to\infty$, and a horseshoe
has positive topological entropy, say larger than some constant~$\nuh_*$ (see Proposition~\ref{P.lsch}),
Equation~\eqref{eq.entropy} implies that the topological
entropy of~$u^L$ is at least $\nuh_* L^{1/2}$.

Finally, we prove the statement about the expected values. As above, we use the functional
$\Phi_r(w)$ to denote the number of different objects (horseshoes, solid tori or periodic orbits). The case of zeros will be considered later. Note that,
since $\Phi_r$ is lower semicontinuous, and $\mu^{L,z}$ converges
weakly to~$\mu_u$ as $L\to\infty$ by Lemma~\ref{L.convmu}, it is standard that~\cite[Exercise 2.6]{Bill}
\[
\liminf_{L\to\infty}\bE^{L,z}\frac{\Phi_r}{|B_r|}\geq \bE
\frac{\Phi_r}{|B_r|}\geq \eta>0\,,
\]
where we have picked some fixed, large enough $r$. Here we have used
the asymptotics in~$\RR^3$, given by Theorem~\ref{T.R3b}, to infer that the last
expectation is positive if~$r$ is large. Notice that the constant $\eta$ depends on $[\gamma],[\cT],\cI$ or $\cJ$ depending on the functional the we are considering, but we shall not write this dependence explicitly. Furthermore, as the
distribution of the measure~$\mu^{L,z}_R$ is in fact independent of~$z$
by Proposition~\ref{P.covL}, this ensures that there is some~$L_0$
independent of~$z$ such that
\[
\bE^{L,z}\frac{\Phi_r}{|B_r|}> \frac\eta2
\]
for all admissible~$L>L_0$ and all $z\in\TT^3$.

Now, given any admissible $L>L_0$, it is standard that we can cover the torus~$\TT^3$ by
balls $\{ B_{r_L}(z^a): 1\leq a\leq A_L\}$ of radius $r_L:=2r/L^{1/2}$ centered at $z^a\in \TT^3$ such
that the smaller balls $B_{r_L/2}(z^a)$ are pairwise disjoint. This
implies that $A_L\geq c L^{\frac32}$ for some dimensional constant~$c$. The
expected value of, say, the number of horseshoes of $u^L$ in $\TT^3$ can then be controlled as follows, for any admissible~$L>L_0$:
\begin{align*}
\frac{\bE^L N^{\mathrm h}}{L^{3/2}} &\geq \sum_{a=1}^{A_L} \frac{|B_r|}{L^{3/2}} \,
                              \bE^{L,z^a}
                              \frac{\Phi_r
                              }{|B_r|}\\
  &\geq \frac{c|B_r|\eta}2>\nu_*
\end{align*}
for some positive constant $\nu_*$ independent of~$L$. An analogous estimate holds for the expected value $\bE^L N^{\mathrm o}([\ga])$.

To estimate the volume of ergodic invariant tori isotopic to~$\cT$ we can proceed as follows. For any admissible~$L>L_0$ we have:
\begin{align*}
\bE^L V^{\mathrm{t}}([\cT]) &\geq \sum_{a=1}^{A_L} |B_{r_L/2}| \,
                              \bE^{L,z^a}
                              \frac{V^{\mathrm{t}}(r;[\cT],\cJ)
                              }{|B_r|}\\
  &\geq \sum_{a=1}^{A_L} |B_{r_L/2}| \, V_0\, \bE^{L,z^a} \frac{\Phi_r
    }{|B_r|}\\
  &\geq \frac{V_0\eta}2\sum_{a=1}^{A_L} |B_{r_L/2}| >\nu^{\mathrm{t}}_*([\cT])
\end{align*}
for some positive constant $\nu^{\mathrm{t}}_*([\cT])$ independent of~$L$. Here we have
used that the balls $B_{r_L/2} (z^a)$ are pairwise disjoint and the sum of their volumes
is, by construction, larger than~$|\TT^3|/8$.

Lastly, in the following lemma we consider the case of zeros:

\begin{lemma}
$\displaystyle \bE^L(L^{-\frac32}\NzL)\to(2\pi)^3\nuz$ as $L\to\infty$ through admissible values.
\end{lemma}

\begin{proof}
  Let us use the notation
\[
Q_R:=(-R\pi,R\pi)\times (-R\pi,R\pi)\times(-R\pi,R\pi)
\]
for the open cube of side~$2\pi R$ in~$\RR^3$ and call $N_{u^L}^{\mathrm{z},*}$ the
number of zeros of~$u^L$ (or rather of its periodic lift to~$\RR^3$) that are contained in~$Q_1$. By Bulinskaya's lemma~\cite[Proposition 6.11]{AW09},
with probability~1 the zero set of~$u^L$ is nondegenerate (and hence a
finite set of points) and the lift of $u^L$ does not have any zeros on the
boundary $\pd Q_1$. Therefore, for any positive integer~$R$,
\[
\NzL= N_{u^L}^{\mathrm{z},*}
\]
almost surely. In particular, both quantities have the same expectation.

Let us now take some small positive real~$r$ and denote by $\NzL(y,r)$
the number of zeros of $u^L$  (or rather of its lift to~$\RR^3$) that are contained in the ball $B_r(y)$. The argument we used to prove
the estimate for $\cN(R;\cX)$ in Lemma~\ref{L.sandwich} (starting now
from the number of zeros in~$Q_1$ instead of in~$B_R$) shows that
\[
\int_{Q_{1-r}}\frac{\NzL(z,r)}{|B_r|}\, dz\leq
N_{u^L}^{\mathrm{z},*}\leq \int_{Q_{1+r}}\frac{\NzL(z,r)}{|B_r|}\, dz\,.
\]

Note now that
\[
\int_{Q_{1\pm r}}\frac{\NzL(z,r)}{|B_r|}\, dz=L^{\frac32}\int_{Q_{1\pm r}}\frac{\NzLz(rL^{1/2})}{|B_{rL^{1/2}}|}\, dz\,.
\]
The expected value of this quantity is
\begin{align*}
\bE^L \int_{Q_{1\pm r}}\frac{\NzLz(rL^{1/2})}{|B_{rL^{1/2}}|}\,
  dz&=\int_{Q_{1\pm
      r}}\frac{\bE^{L,z}\NzLz(rL^{1/2})}{|B_{rL^{1/2}}|}\, dz\\
  &=|Q_{1\pm r}|\frac{\bE^{L,z}\NzLz(rL^{1/2})}{|B_{rL^{1/2}}|}\,.
\end{align*}
To pass to the second line we have used that the expected value inside
the integral is independent of the point~$z$ by
Proposition~\ref{P.covL}; in particular, this value is independent of
the point~$z$ one considers.

We can now argue just as in the case of~$\RR^3$, discussed in detail
in Subsection~\ref{S.zeros}, so we will just sketch the arguments and
refer to that subsection for the notation. The Kac--Rice formula ensures
\begin{align*}
\frac{\bE^{L,z}\NzLz(rL^{1/2})}{|B_{rL^{1/2}}|}&= (2\pi)^{-\frac32}\bE^{L,z}\big(\big\{
                                                 | \det \nabla
                                                 u^{L,z}(0)|: u^{L,z}(0)=0\big\}\big)\,,
\end{align*}
and this conditional expectation can be transformed into an
unconditional one just as in the proof of Lemma~\ref{L.bE}:
\begin{align*}
  \frac{\bE^{L,z}\NzLz(rL^{1/2})}{|B_{rL^{1/2}}|}&=(2\pi)^{-3/2}\bE^{L,z}(|\det\zeta^{L,z}|)\\
  &= \frac{(2\pi)^{-3/2}}{(2\pi)^{5/2}(\det \Si'^{L,z})^{1/2}}\int_{\RR^5}Q^{L,z}(\zeta')\,
    e^{-\frac12 \zeta'\cdot (\Si'^{L,z})^{-1}\zeta'}\, d\zeta'\\
  &=:\nu^{\mathrm{z},L,z}\,.
\end{align*}
The fact that the covariance matrix of $u^{L,z}$ converges to that
of~$u$ as~$L\to\infty$ by Proposition~\ref{P.covL} implies that
\[
  \lim_{L\to\infty}\nu^{\mathrm{z},L,z}=\nuz\,.
  \]

  Hence, writing the aforementioned sandwich estimate as
  \[
|Q_{1-r}|\nu^{\mathrm{z},L,z}\leq \frac{\bE^L\NzL}{L^{3/2}}\leq |Q_{1+r}|\nu^{\mathrm{z},L,z}
  \]
  and letting $L\to\infty$ and then $r\to0$, we infer that
  \[
\lim_{L\to\infty}\frac{\bE^L\NzL}{L^{3/2}}= |Q_1|\nuz=(2\pi)^3\nuz\,.
  \]
  The lemma follows.
\end{proof}

Theorem~\ref{T.torus} is then proven.

\section*{Acknowledgements}

The authors are deeply indebted to Alejandro Luque for the numerical computation of the Melnikov constants in Section~\ref{S.chaos}. A.E.\ is supported by the ERC Starting Grant~633152. D.P.-S.
is supported by the grants MTM-2016-76702-P (MINECO/FEDER) and Europa Excelencia EUR2019-103821 (MCIU). A.R.\ is supported by the grant MTM-2016-76702-P (MINECO/FEDER). This work is supported in part by the ICMAT--Severo Ochoa grant
SEV-2015-0554 and the CSIC grant 20205CEX001.

\appendix

\section{Fourier-theoretic characterization of Beltrami fields}
\label{A.Fourier}

For the benefit of the reader, in this appendix we describe what
polynomially bounded Beltrami fields look like in Fourier space. As
Beltrami fields are a particular class of vector-valued monochromatic
waves, it is convenient to start the discussion by considering
polynomially bounded solutions to the Helmholtz equation
\[
\De F + F=0\,.
\]
As before, we consider the case of monochromatic waves on~$\RR^3$, but
the analysis applies essentially verbatim to any other dimension. The Fourier transform of this equation shows that
\[
(1-|\xi|^2)\hF(\xi)=0\,,
\]
so the support of $\hF$ must be contained in the unit sphere, $\S$. In
spherical coordinates $\rho:=|\xi|\in\RR^+$ and~$\om:=\xi/|\xi|\in \S$, it is standard that
this is equivalent to saying that~$\hF$ is a finite sum of the form
\[
\hF= \sum_{n=1}^N F_n(\om)\, \de^{(n)}(\rho-1)\,.
\]
Here $\de^{(n)}$ is the $n^{\mathrm{th}}$ derivative of the Dirac
measure and $F_n$ is a distribution on the sphere, so $F_n\in
H^{s_n}(\S)$ for some $s_n\in\RR$ (because any compactly supported distribution is in a Sobolev space, possibly of negative order). Note that~$F$ is real valued if and
only if the functions~$F_n$ are Hermitian. Of course, there are also
monochromatic waves that are not polynomially bounded, such as $F:=
e^{x_1}\cos(\sqrt 2 \, x_2)$.

A classical result due to
Herglotz~\cite[Theorem 7.1.28]{Hor15} ensures that if~$F$ is a
monochromatic wave
with the sharp fall off at infinity, i.e., such that
\[
\limsup_{R\to\infty}\frac1R \int_{B_R}F^2\, dx<\infty\,,
\]
then there is a Hermitian vector-valued function $f\in L^2(\S)$ such that
$\hF = f\, \de(\rho-1)$. Furthermore, the value of the above limit is in the
interval $[C_1\|F\|_{L^2(\S)},C_2\|F\|_{L^2(\S)}]$ for some
constants~$C_1,C_2$. This bound means that, on an average sense,
$|F(x)|$ decays as $C/|x|$. The prime example of this behavior is
given by $f=1$, which corresponds to $F(x)=c |x|^{-1/2} J_{1/2}(|x|)$.

The expression~\eqref{random} corresponds to the case $N=0$ above,
since the function~$F_0$ with $\hF_0=f(\om)\, \de(\rho-1)$ is precisely
\[
F_0(x)=\int_\S e^{ix\cdot\om} f(\om)\, d\si(\om)\,.
\]
Also, if $f\in H^{-k}(\S)$ with $k\geq0$ but not necessarily in~$L^2(\S)$, the
function~$F_0$ is bounded as~\cite[Appendix~A]{random}
\begin{equation}\label{boundu}
\sup_{R>0} \frac1R\int_{B_R}\frac{F_0(x)^2}{1+|x|^{2k}}\, dx\leq C\|f\|_{H^{-k}(\S)}\,.
\end{equation}
Hence in this case, $F_0$ is bounded, on an average sense, by
$C|x|^{k-1}$. Therefore, if $f\in H^{-1}(\S)$, $F_0$ is uniformly bounded in average sense.

If $f$ is a Gaussian random field, as considered in the
Nazarov--Sodin theory (see Equation~\eqref{random1}), we showed in
Proposition~\ref{P.conv} that $f$ is almost surely in~$H^{-1-\de}(\S)$ for
all~$\de>0$ and not in~$L^2(\S)$. This behavior morally corresponds to
functions that are bounded on a average sense but do not decay at
infinity, as illustrated by the
function $F_0:=\cos x_1$ generated by $f:= \frac12[\de_{\xi_+}(\xi)+
\de_{\xi_-}(\xi)]$. This is the kind of behavior one needs to describe the expected
local behavior of a high energy eigenfunction on a compact manifold as
one zooms in at a given point.

The monochromatic wave defined as $\hF_n:= f(\om)\,\de^{(n)}(\rho-1)$ reads, in
physical space, as
\begin{align*}
F_n(x)=\int_{\S}\int_0^\infty e^{i\rho x\cdot \om} f(\om)\,\rho^2\,
        \de^{(n)}(\rho-1)\, d\rho\, d\si(\om)
  =(-1)^n \int_\S f(\om)\, \pd_\rho^n|_{\rho=1}(\rho^2 e^{i\rho
    x\cdot \om})\, d\si(\om)\,.
\end{align*}
Note that the $n^{\mathrm{th}}$ derivative term involves an
$n^{\mathrm{th}}$ power of~$x$. Therefore, using the
bound~\eqref{boundu}, one easily finds that $F_n$ is bounded on
average as $C|x|^{n+k-1}$ if $f\in H^{-k}(\S)$; explicit examples with
this growth can be easily constructed by taking $f$ to be either a
constant for $k=0$ or the $(k-1)^{\mathrm{th}}$~derivative of the
Dirac measure for $k\geq1$. Consequently, picking~$f$ as
in~\eqref{random1}, the bound~\eqref{boundu} morally leads to thinking
of~$F_n$ as a function that grows as $|x|^n$ at infinity, which
cannot be the localized behavior of an eigenfunction. This is the
rationale for defining a random monochromatic wave as
in~\eqref{random1}-\eqref{random2}. In this direction, let us recall
that the relation between random monochromatic waves and zoomed-in
high energy eigenfunctions on a various compact manifolds is an
influential long-standing conjecture of Berry~\cite{Berry}. A precise
form of this relation has been recently established in the case of the
round sphere and of the flat torus~\cite{NS09,NS16,Roz17}, which
heuristically shows that~\eqref{random1}-\eqref{random2} is indeed the
proper definition of random monochromatic waves for this purpose.

The reasoning leading to the definition of a random Beltrami field
as~\eqref{random} is completely analogous, and the fact that one can
relate Gaussian random Beltrami fields on~$\RR^3$ to high-frequency
Beltrami fields on the torus just as in the case of
the Nazarov--Sodin theory heuristically ensures that this is indeed
the appropriate definition. For completeness, let us record that, just as in the case of
monochromatic random waves, the Fourier transform of a polynomially
bounded Beltrami field~$u$ is a finite sum of the form
\[
\hu = \sum_{n=1}^N f_n(\om)\, \de^{(n)}(\rho-1)\,,
\]
where now $f_n$ is a Hermitian $\CC^3$-valued distribution
on~$\S$. For $u$~to be a Beltrami field, there is an additional constraint on~$f_n$ coming from the fact that
not every distribution supported on~$\S$ satisfies the equation
$i\xi\times \hu(\xi)= \hu(\xi)$. A straightforward computation shows
that this constraint amounts to imposing that
\[
\sum_{n=j}^N \binom nj \al_{n-j,2} f_n(\om)= i\om\times \sum_{n=j}^N
\binom nj \al_{n-j,3} f_n(\om)
\]
on~$\S$ for all $0\leq j\leq N$. Here $\al_{k,l}:= \prod_{m=0}^{k-1} (l-m)$
with the convention that $\al_{0,l}:=1$. To see this, it suffices to
note that the action of $\hu$ and $i\xi\times \hu$ on a vector field
$w\in C^\infty_c(\RR^3,\RR^3)$ is
\begin{align*}
\langle \hu,w \rangle &= \sum_{n=0}^N (-1)^n\int_\S f_n(\om)\cdot
                        \pd_{\rho}^n|_{\rho=1} \left[\rho^2
                        w(\rho\om)\right]\, d\si(\om)\,,\\
  \langle i\xi\times \hu,w \rangle &= \sum_{n=0}^N (-1)^n\int_\S
                                     i\om\times f_n(\om)\cdot
                        \pd_{\rho}^n|_{\rho=1} \left[\rho^3
                        w(\rho\om)\right]\, d\si(\om)\,,
\end{align*}
expand the $n^{\mathrm{th}}$ derivative using the binomial formula and
note that $\al_{k,l}$ is the $k^{\mathrm{th}}$ derivative of $\rho^l$
at $\rho=1$.

\end{document}